\documentclass[12pt]{article}

\usepackage{amsmath}
\usepackage{amsfonts}
\usepackage{amssymb}
\usepackage{wasysym}
\usepackage{amsthm}
\usepackage[margin=1in]{geometry}

\title{Intertwining operators among modules for affine Lie algebra and lattice 
vertex operator algebras which respect integral forms}
\author{Robert McRae}
\date{}

    \theoremstyle{definition}\newtheorem{rema}{Remark}[section]
    \theoremstyle{plain}\newtheorem{propo}[rema]{Proposition}
    \newtheorem{theo}[rema]{Theorem}
    \newtheorem{defi}[rema]{Definition}
    \newtheorem{lemma}[rema]{Lemma}
    \newtheorem{corol}[rema]{Corollary}
    \theoremstyle{definition}

\begin{document}
\bibliographystyle{alpha}
\maketitle

\newcommand{\ghat}{\widehat{\mathfrak{g}}}
\newcommand{\Z}{\mathbb{Z}}
\newcommand{\gvmu}{V_{\widehat{\mathfrak{g}}}(\ell,U)}
\newcommand{\imu}{L_{\widehat{\mathfrak{g}}}(\ell,U)}
\newcommand{\gvmzero}{V_{\widehat{\mathfrak{g}}}(\ell,0)}
\newcommand{\imzero}{L_{\widehat{\mathfrak{g}}}(\ell,0)}
\newcommand{\gvmlambda}{V_{\widehat{\mathfrak{g}}}(\ell,L_\lambda)}
\newcommand{\imlambda}{L_{\widehat{\mathfrak{g}}}(\ell,L_\lambda)}
\numberwithin{equation}{section}

\begin{abstract}
\noindent We define an integral intertwining operator among modules for a 
vertex 
operator algebra to be an intertwining operator which respects integral forms 
in 
the modules, and we show that an intertwining operator is integral if it is 
integral when restricted to generators of the integral forms in the modules. We 
apply this result to classify integral intertwining operators which respect 
certain natural integral forms in modules for affine Lie algebra and lattice 
vertex operator algebras.
\end{abstract}

\section{Introduction}

Integral forms for lattice vertex algebras were first introduced by Borcherds 
in 
\cite{B} and were also studied in \cite{P}. They have been used in the modular 
moonshine program initiated by Borcherds and Ryba (\cite{R}, \cite{BR1}, 
\cite{BR2}, \cite{GL}). Recently, integral forms in vertex algebras related to 
lattice vertex algebras, including automorphism fixed points and the moonshine 
module $V^\natural$, have been constructed in \cite{DG} and \cite{GL}, and 
their 
automorphism groups have been studied. The representation theory of integral 
vertex algebras based on affine Lie algebras and lattices was studied in 
\cite{M2}: in particular, integral forms were constructed in both vertex 
algebras and their modules. It was also determined that the graded 
$\mathbb{Z}$-duals of integral forms in modules for these algebras form natural 
integral forms in the contragredient modules.

In the present paper, we continue to study the representation theory of 
integral 
vertex operator algebras in general, and in particular the representation 
theory 
of integral affine Lie algebra and lattice vertex operator algebras, by showing 
when intertwining operators among modules respect integral forms. Intertwining 
operators play an essential role in the representation theory of vertex 
operator 
algebras. For instance associativity of intertwining operators (\cite{H1}, 
\cite{HLZ3}) and modular invariance for traces of products of intertwining 
operators (\cite{H2}) are crucial for showing that modules for certain vertex 
operator algebras form modular tensor categories (\cite{H3}; see also the 
review 
article \cite{HL}). Thus it is natural to consider intertwining operators among 
modules for vertex algebras over $\mathbb{Z}$. 

Since also vertex algebras over $\Z$ lead to vertex algebras over finite fields 
(and more generally over arbitrary fields of prime characteristic) through 
reducing structure constants mod a prime $p$, intertwining operators among 
modules for vertex algebras over $\Z$ lead to intertwining operators among 
modules for vertex algebras over fields of characteristic $p$. Intertwining 
operators among modules for the Virasoro vertex operator algebra 
$L(\frac{1}{2},0)$ over fields of odd prime characteristic have already been 
studied in \cite{DR}. The present paper will have immediate application to 
intertwining operators among modules for affine Lie algebra and lattice vertex 
operator algebras in prime characteristic.

If $V$ is a vertex operator algebra and $W^{(i)}$ for $i=1,2,3$ are a triple of 
$V$-modules, then an intertwining operator of type 
$\binom{W^{(3)}}{W^{(1)}\,W^{(2)}}$ is a linear map
\begin{equation*}
 \mathcal{Y}: W^{(1)}\otimes W^{(2)}\rightarrow W^{(3)}\{x\}
\end{equation*}
that satisfies a lower truncation axiom, an $L(-1)$-derivative property, and a 
Jacobi identity. Here $W^{(3)}\lbrace x\rbrace$ represents formal series with 
arbitrary complex powers and coefficients in $W^{(3)}$. If $V$ and the modules 
$W^{(i)}$ have integral forms $V_\Z$ and $W^{(i)}_\Z$, respectively, then it is 
natural to consider which intertwining operators among these modules respect 
the 
integral forms. In particular, it is natural to say that an intertwining 
operator of type $\binom{W^{(3)}}{W^{(1)}\,W^{(2)}}$ is integral with respect 
to 
the integral forms $W^{(i)}_\Z$ if
\begin{equation}\label{intwopint}
 \mathcal{Y}(w_{(1)},x)w_{(2)}\in W^{(3)}_\Z\lbrace x\rbrace
\end{equation}
for $w_{(1)}\in W^{(1)}_\Z$ and $w_{(2)}\in W^{(2)}_\Z$.

The main general result of this paper is that in order to check that an 
intertwining operator of type $\binom{W^{(3)}}{W^{(1)}\,W^{(2)}}$ is integral, 
it is enough to check that \eqref{intwopint} holds when $w_{(1)}$ and $w_{(2)}$ 
come from generating sets of $W^{(1)}_\Z$ and $W^{(2)}_\Z$, respectively. This 
result fits with the general philosopy of \cite{M2} that integral forms of 
vertex operator algebras and modules are often best studied using generating 
sets. In the remainder of the paper, we apply the general theorem to vertex 
operator algebras coming from affine Lie algebras and lattices; natural 
integral 
forms in modules for these algebras were obtained in \cite{M2}.

Interestingly, the examples studied in this paper suggest another notion of integrality for intertwining operators different from \eqref{intwopint}. Specifically,
suppose that $V$ is an affine Lie algebra or lattice vertex operator algebra 
with modules $W^{(i)}$ for $i=1,2,3$, and suppose that $V_\Z$ and $W^{(i)}_\Z$ 
are the integral forms in $V$ and $W^{(i)}$, respectively, that were 
constructed in \cite{M2}. Then we are able to find and classify non-zero 
intertwining operators which are integral with respect to  $W^{(1)}_\Z$, 
$W^{(2)}_\Z$, and $((W^{(3)})'_\Z)'$. Here $((W^{(3)})'_\Z)'$ is the graded 
$\Z$-dual of the natural integral form in the contragredient module 
$(W^{(3)})'$; this is an integral form of $W^{(3)}$ which is often larger than 
$W^{(3)}_\Z$. In other words, we are able to classify the intertwining 
operators 
of type $\binom{W^{(3)}}{W^{(1)}\,W^{(2)}}$ which satisfy
\begin{equation}\label{intwopint2}
 \langle w_{(3)}',\mathcal{Y}(w_{(1)},x)w_{(2)}\rangle\in\Z\lbrace x\rbrace
\end{equation}
when $w_{(1)}$, $w_{(2)}$, and $w_{(3)}'$ come from the natural integral forms 
of $W^{(1)}$, $W^{(2)}$, and $(W^{(3)})'$, respectively. This suggests that 
\eqref{intwopint2} is a more fundamental integrality condition than \eqref{intwopint}, a notion supported by the fact that \eqref{intwopint2} corresponds to an integrality condition on physically relevant correlation functions in conformal field theory.

The remainder of this paper is structured as follows. In the next section we 
recall the notions of integral form in a vertex operator algebra and in a 
module 
for a vertex operator algebra from \cite{M2} and recall some results on 
integral 
forms in contragredient modules. In Section 3, we recall the definition of 
intertwining operator among modules for a vertex operator algebra and define 
what it means for an intertwining operator to be integral with respect to 
integral forms in the modules. We also prove our general result that an 
intertwining operator is integral if it is integral when restricted to 
generators. Finally, we derive a way to use intertwining operators to identify 
a 
module with its contragredient; this will be used in the last section when we 
study integral intertwining operators for lattice vertex algebras.

In Section 4 we recall the construction of the level $\ell$ generalized Verma 
module vertex operator algebra $\gvmzero$ based on an affine Lie algebra 
$\widehat{\mathfrak{g}}$, where $\mathfrak{g}$ is a finite-dimensional simple 
complex Lie algebra. We also recall the classification of irreducible 
$\gvmzero$-modules, including the irreducible quotient $\imzero$ of 
$\gvmzero$. 
We recall from \cite{FZ} (and \cite{Li2}) the characterization of intertwining 
operators among $\imzero$-modules for nonnegative integral level $\ell$. In 
Section 5, we recall the natural integral forms in $\gvmzero$- and 
$\imzero$-modules from \cite{M2} and classify integral intertwining operators 
among $\imzero$-modules when $\ell$ is a nonnegative integer.

In Section 6, we recall the construction of the vertex algebra $V_L$ and its 
modules from an even nondegenerate lattice $L$ and recall from \cite{DL} the 
construction of intertwining operators among $V_L$-modules. In Section 7, we 
recall from \cite{M2} the natural integral form in a $V_L$-module and use 
nondegenerate invariant bilinear pairings to explicitly describe the graded 
$\Z$-dual of the natural integral form of a $V_L$-module. Using this, we 
classify integral intertwining operators among $V_L$-modules.

\paragraph{Acknowledgments}
This paper is part of my thesis \cite{M1}, completed at Rutgers University. I 
am 
very grateful to my advisor James Lepowsky for all of his advice and 
encouragement, and to Yi-Zhi Huang for comments on this paper.

\section{Integral forms in vertex operator algebras and modules}

Throughout this paper, we will use standard terminology and notation from the 
theory of vertex operator algebras, sometimes without comment. In particular, 
we 
will use the notion of vertex algebra as defined in \cite{B} and the notion of 
vertex operator algebra as defined in \cite{FLM} (see also \cite{LL}). We will 
also use the notion of module for a vertex (operator) algebra as defined in 
\cite{LL} (where modules for a vertex operator algebra have a conformal weight 
grading by $\mathbb{C}$). We recall that a vertex operator algebra and its 
modules admit representations of the Virasoro Lie algebra, and as usual we use 
$L(n)$ to denote the action of Virasoro algebra operators on a vertex operator 
algebra and its modules.

Vertex algebras over $\Z$ make sense, and we refer to them as vertex rings. As 
in \cite{M2}, we define an integral form in a vertex operator algebra as 
follows:
\begin{defi}
 An \textit{integral form} of a vertex operator algebra $V$ is a vertex subring 
$V_\Z\subseteq V$ that is an integral form of $V$ as a vector space and that is 
compatible with the conformal weight grading of $V$:
 \begin{equation*}\label{compatibility}
  V_\Z=\coprod_{n\in\mathbb{Z}} V_n\cap V_\Z,
 \end{equation*}
where $V_n$ is the weight space with $L(0)$-eigenvalue $n$.
\end{defi}
\begin{rema}
 Equivalently, an integral form $V_\Z$ of $V$ is the $\mathbb{Z}$-span of a basis for $V$ 
which contains the vacuum $\mathbf{1}$, is closed under vertex algebra 
products, 
and is compatible with the conformal weight gradation.
\end{rema}
\begin{rema}
 In \cite{DG}, the definition of an integral form $V_\Z$ of $V$ is somewhat 
different: compatibility with the weight gradation is replaced by the 
requirement 
that an integral multiple of the conformal vector $\omega$ be in $V_\Z$. In the examples studied in this paper, both conditions hold.
\end{rema}

Once we have fixed an integral form $V_\Z$ of a vertex operator algebra $V$, it 
is easy to define the notion of an integral form in a $V$-module:
\begin{defi}
 An \textit{integral form} in a $V$-module $W$ is a $V_\Z$-submodule 
$W_\Z\subseteq W$ that is an integral form of $W$ as a vector space and that is 
compatible with the conformal weight grading of $W$:
 \begin{equation*}\label{modcompatibility}
  W_\mathbb{Z}=\coprod_{h\in\mathbb{C}} W_h\cap W_\Z,
 \end{equation*}
where $W_h$ is the weight space with $L(0)$-eigenvalue $h$.
\end{defi}
\begin{rema}
 Equivalently, an integral form $W_\Z$ of $W$ is the $\mathbb{Z}$-span of a basis for $W$ 
which is preserved by vertex operators from $V_\Z$ and is compatible with the 
conformal weight gradation.
\end{rema}
\begin{rema}
 Note that the notion of an integral form of $W$ depends on the integral form 
$V_\Z$ of $V$ that is used.
\end{rema}

Here we note that our approach to studying integral forms in vertex operator 
algebras and their modules uses generating sets. If $V_\Z$ is an integral form 
of a vertex operator algebra $V$, we say that $S\subseteq V_\Z$ generates 
$V_\Z$ 
if $V_\Z$ is the smallest vertex subring of $V$ containing $S$. Similarly, if 
$W$ is a $V$-module with integral form $W_\Z$, we say that $T\subseteq W_\Z$ 
generates $W_\Z$ if $W_\Z$ is the smallest $V_\Z$-submodule of $W$ containing 
$T$.
% Since our approach to studying integral forms in vertex operator 
% algebras heavily uses generating sets, we recall the following 
% proposition from \cite{M2}:
% \begin{propo}\label{zgen}
%  Suppose $V$ is a vertex algebra; for a subset $S$ of $V$, denote by 
% $\left\langle S\right\rangle_\mathbb{Z}$ the vertex subring
% generated by $S$. Then  $\left\langle S\right\rangle_\mathbb{Z}$ is the 
% $\mathbb{Z}$-span of coefficients of products of the form
% \begin{equation}\label{zspan}
%  Y(u_1,x_1)\ldots Y(u_k,x_k)\mathbf{1}
% \end{equation}
% where $u_1,\ldots u_k\in S$. Moreover, if $W$ is a $V$-module, the 
% $\left\langle S\right\rangle_\mathbb{Z}$-submodule generated by a subset 
% $T$ of $W$ is the $\mathbb{Z}$-span of coefficients of products of the 
% form
% \begin{equation}
%  Y(u_1,x_1)\ldots Y(u_k,x_k)w
% \end{equation}
% where $u_1,\ldots u_k\in S$ and $w\in T$.
% \end{propo}

For our study of intertwining operators in this paper, we will also need to use 
contragredients of modules for a vertex operator algebra $V$. Recall from 
\cite{FHL} that if $W$ is a $V$-module, the contragredient $V$-module $W'$ is 
the graded dual
\begin{equation*}
 W'=\coprod_{h\in\mathbb{C}} W_h^*
\end{equation*}
equipped with the vertex operator defined by
\begin{equation*}
 \langle Y_{W'}(v,x)w',w\rangle=\langle w', Y^o_W(v,x)w\rangle
\end{equation*}
for $v\in V$, $w\in W$, and $w'\in W'$. Here $Y^o_W$ denotes the 
\textit{opposite vertex operator}
\begin{equation}\label{oppopdef}
Y^o_W(v,x) = Y(e^{x L(1)} (-x^{-2})^{L(0)} v,x^{-1}).
\end{equation}
Recall also from \cite{FHL} that for any $V$-module $W$, $W\cong (W')'$, so 
that it makes sense to refer to contragredient pairs of $V$-modules.

If $V_\Z$ is an integral form of $V$ and $W$ is a $V$-module with integral 
form 
$W_\Z$, then the graded $\mathbb{Z}$-dual of $W_\Z$ defined by
\begin{equation*}
 W'_\Z=\lbrace w'\in W'\,\vert\,\langle 
w',w\rangle\in\Z\,\,\mathrm{for}\,\,w\in 
W_\Z\rbrace
\end{equation*}
is an integral form of $W'$ as a vector space. The following two propositions 
from \cite{M2} (see also Lemma 6.1, Lemma 6.2, and Remark 6.3 in \cite{DG}) 
show 
when $W'_\Z$ is a $V_\Z$-module:
\begin{propo}
 Suppose $V_\mathbb{Z}$ is preserved by $\frac{L(1)^n}{n!}$ for $n\geq 0$.
Then $W'_\mathbb{Z}$ is preserved by the action of $V_\mathbb{Z}$.
\end{propo}
\begin{propo}
 If $V_\mathbb{Z}$ is generated by vectors $v$ such that $L(1)v=0$, then 
$V_\mathbb{Z}$ is preserved by $\frac{L(1)^n}{n!}$ for $n\geq 0$.
\end{propo}

\begin{rema}
 If an integral form $V_\Z$ is invariant under the operators $\frac{L(1)^n}{n!}$ for $n\geq 0$, then in fact $V_\Z$ is invariant under the integral form of the universal enveloping algebra of $\mathfrak{sl}_2=\mathrm{span}\lbrace L(-1), L(0), L(1)\rbrace$ with basis given by the ordered products
 \begin{equation*}
  \dfrac{L(-1)^k}{k!}\cdot\binom{L(0)}{m}\cdot\dfrac{L(1)^n}{n!}
 \end{equation*}
where $k,m,n\geq 0$. Note that $V_\Z$ is automatically invariant under the operators $\frac{L(-1)^k}{k!}$ because $e^{xL(-1)}v=Y(v,x)\mathbf{1}$ for any $v\in V_\Z$ and $V_\Z$ is invariant under the operaters $\binom{L(0)}{m}$ by compatibility of $V_\Z$ with the weight gradation. But note that we cannot expect all of these operators to preserve integral forms of $V$-modules because for instance $V$-modules may not be graded by integers.
\end{rema}

Sometimes a vertex operator algebra $V$ is equivalent as a $V$-module to its 
contragredient $V'$. From \cite{FHL}, this is the case if and only if $V$ has a 
nondegenerate bilinear form $(\cdot,\cdot)$ that is \textit{invariant} in the 
sense that
\begin{equation*}
 (Y(u,x)v,w)=(v,Y^o(u,x)w)
\end{equation*}
for all $u,v,w\in V$. From \cite{Li}, we know that the space of (not 
necessarily 
nondegenerate) invariant bilinear forms on $V$ is linearly isomorphic to 
$V_0/L(1)V_1$, where as usual $V_n$ for $n\in\Z$ represents the weight space 
with $L(0)$-eigenvalue $n$. If $V$ is a simple vertex operator algebra, then 
any non-zero invariant bilinear form is necessarily nondegenerate.

\section{Integral intertwining operators}
For any vector space or $\mathbb{Z}$-module $V$, let $V\lbrace x\rbrace$ 
denote 
the space of formal series with complex powers and coefficients in $V$. We 
recall the definition of intertwining operator among a triple of modules for a 
vertex operator algebra (see for instance \cite{FHL}):
\begin{defi}
 Suppose $V$ is a vertex operator algebra and $W^{(1)}$,  $W^{(2)}$ and 
$W^{(3)}$ are $V$-modules. An \textit{intertwining operator} of type 
$\binom{W^{(3)}}{W^{(1)}\,W^{(2)}}$ is a linear map
 \begin{eqnarray*}
\mathcal{Y}: W^{(1)}\otimes W^{(2)}&\to& W^{(3)}\{x\},
\\w_{(1)}\otimes w_{(2)}&\mapsto &\mathcal{Y}(w_{(1)},x)w_{(2)}=\sum_{n\in
{\mathbb C}}(w_{(1)})_n
w_{(2)}x^{-n-1}\in W^{(3)}\{x\}
\end{eqnarray*}
satisfying the
following conditions:
\begin{enumerate}

\item  {\it Lower truncation}: For any $w_{(1)}\in W^{(1)}$, $w_{(2)}\in 
W^{(2)}$ and $n\in\mathbb{C}$,
\begin{equation*}\label{log:ltc}
(w_{(1)})_{n+m}w_{(2)}=0\;\;\mbox{ for }\;m\in {\mathbb
N} \;\mbox{ sufficiently large.}
\end{equation*}

\item The {\it Jacobi identity}:
\begin{align}\label{intwopjac}
 x^{-1}_0\delta \left(\frac{x_1-x_2}{x_0}\right) 
Y_{W^{(3)}}(v,x_1)\mathcal{Y}(w_{(1)},x_2) & - 
x^{-1}_0\delta\left(\frac{x_2-x_1}{-x_0}\right)\mathcal{Y}(w_{(1)},x_2)Y_{W^{(2)
}}(v,x_1)\nonumber\\
& =  x^{-1}_2\delta \left(\frac{x_1-x_0}{x_2}\right) 
\mathcal{Y}(Y_{W^{(1)}}(v,x_0)w_{(1)},x_2)
\end{align}
for $v\in V$ and $w_{(1)}\in W^{(1)}$.

\item The  $L(-1)$-derivative property: for any
$w_{(1)}\in W^{(1)}$,
\begin{equation}\label{intwopderiv}
\mathcal{Y}(L(-1)w_{(1)},x)=\frac{d}{dx}\mathcal{Y}(w_{(1)},x).
\end{equation}
\end{enumerate}
\end{defi}
\begin{rema}
 We use $V^{W^{(3)}}_{W^{(1)} W^{(2)}}$ to denote the vector space of 
intertwining operators of type $\binom{W^{(3)}}{W^{(1)}\,W^{(2)}}$, and the 
dimension of $V^{W^{(3)}}_{W^{(1)} W^{(2)}}$ is the \textit{fusion rule} 
$N^{W^{(3)}}_{W^{(1)} W^{(2)}}$
\end{rema}
\begin{rema}
 Note that if $W$ is a $V$-module, the vertex operator $Y_W$ is an 
intertwining 
operator of type $\binom{W}{V\,W}$. The Jacobi identity 
(\ref{intwopjac}) with $v=\omega$ implies that for $w_{(1)}\in W^{(1)}$ and 
$w_{(2)}\in W^{(2)}$ both homogeneous,
 \begin{equation}\label{intwopdeg}
  \mathrm{wt}\,(w_{(1)})_n w_{(2)}=\mathrm{wt}\,w_{(1)}+\mathrm{wt}\,w_{(2)}-n-1
 \end{equation}
for any $n\in\mathbb{C}$.
\end{rema}
\begin{rema}
 If $W^{(1)}$,  $W^{(2)}$ and $W^{(3)}$ are irreducible $V$-modules, then 
there 
are complex numbers $h_i$ for $i=1,2,3$ such that the conformal weights of 
$W^{(i)}$ are contained in $h_i+\mathbb{N}$ for each $i$. If $\mathcal{Y}$ is 
an 
intertwining operator of type $\binom{W^{(3)}}{W^{(1)}\,W^{(2)}}$ and we set 
 \begin{equation*}
  h=h_1+h_2-h_3,
 \end{equation*}
then (\ref{intwopdeg}) implies that we can write
\begin{equation}\label{homogeneousops}
 \mathcal{Y}(w_{(1)},x)w_{(2)}=\sum_{n\in\mathbb{Z}} w_{(1)}(n)w_{(2)}\,x^{-n-h}
\end{equation}
for any $w_{(1)}\in W^{(1)}$ and $w_{(2)}\in W^{(2)}$, where 
$w_{(1)}(n)=(w_{(1)})_{n+h-1}$ for $n\in\mathbb{Z}$. In particular, for 
$w_{(1)}\in W^{(1)}_{h_1}$ and $W^{(2)}_{h_2}$, note that 
$w_{(1)}(0)w_{(2)}\in W^{(3)}_{h_3}$.
\end{rema}

The Jacobi identity for intertwining operators, like the Jacobi identity for 
vertex algebras and modules, implies commutator and iterate formulas. We will 
in 
particular need the iterate formula: for any $v\in V$, $w_{(1)}\in W^{(1)}$, 
and 
$n\in\mathbb{Z}$,
\begin{align}\label{assocint}
 \mathcal{Y}(v_n w_{(1)},x_2)  = & \,\mathrm{Res}_{x_1}(x_1-x_2)^n 
Y_{W^{(3)}}(v,x_1)\mathcal{Y}(w_{(1)},x_2)\nonumber\\
& -\mathrm{Res}_{x_1}(-x_2+x_1)^n 
\mathcal{Y}(w_{(1)},x_2)Y_{W^{(2)}}(v,x_1).
\end{align}
We will also need weak commutativity for intertwining operators, whose proof is 
exactly the same as the proof of weak commutativity for algebras and modules 
(Propositions 3.2.1 and 4.2.1 in \cite{LL}):
\begin{propo}
 Suppose $W^{(1)}$,  $W^{(2)}$ and $W^{(3)}$ are $V$-modules and 
$\mathcal{Y}\in 
V^{W^{(3)}}_{W^{(1)} W^{(2)}}$. Then for any positive integer $k$ such that 
$v_n 
w_{(1)}=0$ for $n\geq k$,
\begin{equation*}
 (x_1-x_2)^k 
Y_{W^{(3)}}(v,x_1)\mathcal{Y}(w_{(1)},x_2)=(x_1-x_2)^k\mathcal{Y}(w_{(1)},x_2)Y_
{W^{(2)}}(v,x_1).
\end{equation*}
\end{propo}
\begin{proof}
 Multiply the Jacobi identity (\ref{intwopjac}) by $x_0^k$ and extract the 
coefficient of $x_0^{-1}$, obtaining
 \begin{align*}
 (x_1-x_2)^k Y_{W^{(3)}}(v,x_1)\mathcal{Y}(w_{(1)},x_2) & 
-(x_1-x_2)^k\mathcal{Y}(w_{(1)},x_2)Y_{W^{(2)}}(v,x_1)\\
 & =\mathrm{Res}_{x_0}\,x_0^k x^{-1}_2\delta\left(\frac{x_1-x_0}{x_2}\right) 
\mathcal{Y}(Y_{W^{(1)}}(v,x_0)w_{(1)},x_2).
\end{align*}
Since $v_n w_{(1)}=0$ for $n\geq k$, the right side contains no negative powers 
of $x_0$, so the residue is $0$.
\end{proof}

We will need the following proposition, which is similar to Proposition 4.5.8 
in 
\cite{LL} and uses essentially the same proof:
\begin{propo}\label{commform}
 Suppose $W^{(1)}$, $W^{(2)}$, and $W^{(3)}$ are $V$-modules and 
$\mathcal{Y}\in 
V^{W^{(3)}}_{W^{(1)} W^{(2)}}$. Then for any $v\in V$, $w_{(1)}\in W^{(1)}$, 
and 
$w_{(2)}\in W^{(2)}$, and for any $p\in\mathbb{C}$ and $q\in\mathbb{Z}$, 
$(w_{(1)})_p v_q w_{(2)}$ is an integral linear combination of terms of the 
form 
$v_r (w_{(1)})_s w_{(2)}$. Specifically, let $k$ and $m$ be nonnegative 
integers 
such that $v_n w_{(1)}=0$ for $n\geq k$ and $v_n w_{(2)}=0$ for $n>m+q$. Then
 \begin{equation*}
 (w_{(1)})_p v_q w_{(2)}=\sum_{i=0}^{m}\sum_{j=0}^{k} 
(-1)^{i+j}\binom{-k}{i}\binom{k}{j} v_{q+i+j} (w_{(1)})_{p-i-j}  w_{(2)}.
 \end{equation*}
\end{propo}
\begin{proof}
 From weak commutativity, we have
 \begin{align}\label{weakcommcalc}
  (w_{(1)})_p v_q w_{(2)} & =\mathrm{Res}_{x_1}\mathrm{Res}_{x_2} x_1^q x_2^p 
\mathcal{Y}(w_{(1)},x_2)Y_{W^{(2)}}(v,x_1)w_{(2)}\nonumber\\
  & = \mathrm{Res}_{x_1}\mathrm{Res}_{x_2} x_1^q x_2^p (-x_2+x_1)^{-k} 
[(x_1-x_2)^{k} \mathcal{Y}(w_{(1)},x_2)Y_{W^{(2)}}(v,x_1)w_{(2)}]\nonumber\\
  & =\mathrm{Res}_{x_1}\mathrm{Res}_{x_2} x_1^q x_2^p (-x_2+x_1)^{-k} 
[(x_1-x_2)^{k}Y_{W^{(3)}}(v,x_1)\mathcal{Y}(w_{(1)},x_2)w_{(2)}].
 \end{align}
These formal expressions are well defined because $Y_{W^{(2)}}(v,x_1)w_{(2)}$ 
is 
lower-truncated, but we cannot remove the brackets from the last expression in 
(\ref{weakcommcalc}). We observe that the term in brackets in 
(\ref{weakcommcalc}) may be written explicitly as
\begin{equation*}
 \sum_{j\in\mathbb{N},\,m\in\mathbb{Z},\,n\in\mathbb{C}} (-1)^j\binom{k}{j}  
v_m 
(w_{(1)})_n w_{(2)} x_1^{k-j-m-1} x_2^{j-n-1}.
\end{equation*}
Meanwhile, only a finite truncation of $x_1^q x_2^p (-x_2+x_1)^{-k}$ 
contributes 
to the residue since $Y_{W^{(2)}}(v,x_1)w_{(2)}$ is lower-truncated. In 
particular, the lowest possible integral power in $x_1^q 
Y_{W^{(2)}}(v,x_1)w_{(2)}$ with a non-zero coefficient is $x_2^{-m-1}$ by 
definition of $m$. Thus we can 
take
\begin{equation*}
 \sum_{i=0}^m (-1)^{k+i}\binom{-k}{i} x_1^{q+i}x_2^{p-k-i}
\end{equation*}
as our truncation of $x_1^q x_2^p (-x_2+x_1)^{-k}$. Then
\begin{align*}
  ( & w_{(1)})_p v_q  w_{(2)}\\
 & =\mathrm{Res}_{x_1,x_2}\sum_{i,j,m,n} (-1)^{k+i+j}\binom{-k}{i}\binom{k}{j} 
v_m (w_{(1)})_n  w_{(2)} x_1^{q+k+i-j-m-1} x_2^{p-k-i+j-n-1}\\
 & =\sum_{i=0}^m\sum_{j=0}^k (-1)^{k+i+j}\binom{-k}{i}\binom{k}{j} v_{q+k+i-j} 
(w_{(1)})_{p-k-i+j}  w_{(2)}\\
 & =\sum_{i=0}^m\sum_{j=0}^k (-1)^{i+j}\binom{-k}{i}\binom{k}{j} v_{q+i+j} 
(w_{(1)})_{p-i-j} w_{(2)},
\end{align*}
where we have changed $j$ to $k-j$ in the last equality.
\end{proof}

 The Jacobi identity for intertwining operators makes sense in the vertex ring 
context for the same reasons that the Jacobi identity for algebras and modules 
does. However, the $L(-1)$-derivative property may not make sense. If 
$V_\mathbb{Z}$ is a vertex ring without a conformal vector $\omega$, then there 
may be no $L(-1)$ operator on $V_\mathbb{Z}$-modules. Additionally, since 
intertwining operators involve complex powers of $x$, the coefficients of the 
derivative of an intertwining operator may not make sense as maps from a 
$V_\mathbb{Z}$-module into a $V_\mathbb{Z}$-module.
 
 However, suppose $V$ is a vertex operator algebra with integral form 
$V_\mathbb{Z}$ and $W^{(1)}$, $W^{(2)}$, and $W^{(3)}$ are $V$-modules with 
integral forms $W^{(1)}_\mathbb{Z}$, $W^{(2)}_\mathbb{Z}$, and 
$W^{(3)}_\mathbb{Z}$, respectively:
 \begin{defi}
  An intertwining operator $\mathcal{Y}\in 
V^{W^{(3)}}_{W^{(1)} W^{(2)}}$ is \textit{integral} with respect to 
$W^{(1)}_\mathbb{Z}$, $W^{(2)}_\mathbb{Z}$, and $W^{(3)}_\mathbb{Z}$ if for any 
$w_{(1)}\in W^{(1)}_\mathbb{Z}$ and $w_{(2)}\in W^{(2)}_\mathbb{Z}$,
 \begin{equation*}
  \mathcal{Y}(w_{(1)},x)w_{(2)}\in W^{(3)}_\mathbb{Z}\lbrace x\rbrace.
 \end{equation*}
 \end{defi}
 \begin{rema}
Note that whether an intertwining operator is integral or not will generally 
depend on the integral forms used for the three $V$-modules. 
\end{rema}
The main result of this section is the following theorem which reduces the 
problem of showing that an intertwining operator is integral to the problem of 
showing that it is integral when restricted to generators of $W^{(1)}$ and 
$W^{(2)}$:
\begin{theo}\label{intwopgen}
 Suppose $V$ is a vertex operator algebra with integral form $V_\mathbb{Z}$ and 
$W^{(1)}$, $W^{(2)}$, and $W^{(3)}$ are $V$-modules with integral forms 
$W^{(1)}_\mathbb{Z}$, $W^{(2)}_\mathbb{Z}$, and $W^{(3)}_\mathbb{Z}$, 
respectively. Moreover, suppose $T^{(1)}$ and $T^{(2)}$ are generating sets for 
$W^{(1)}_\mathbb{Z}$ and $W^{(2)}_\mathbb{Z}$, respectively. If $\mathcal{Y}\in 
V^{W^{(3)}}_{W^{(1)} W^{(2)}}$ satisfies
 \begin{equation*}
  \mathcal{Y}(t_{(1)},x)t_{(2)}\in W^{(3)}_\mathbb{Z}\lbrace x\rbrace
 \end{equation*}
for all $t_{(1)}\in T^{(1)}$, $t_{(2)}\in T^{(2)}$, then $\mathcal{Y}$ is 
integral with respect to $W^{(1)}_\mathbb{Z}$, $W^{(2)}_\mathbb{Z}$, and 
$W^{(3)}_\mathbb{Z}$.
\end{theo}
\begin{proof}
 First let $U^{(2)}_\mathbb{Z}$ be the sublattice of $W^{(2)}_\mathbb{Z}$ 
consisting of vectors $u_{(2)}$ such that
 \begin{equation*}
  \mathcal{Y}(t_{(1)},x)u_{(2)}\in W^{(3)}_\mathbb{Z}\lbrace x\rbrace
 \end{equation*}
for all $t_{(1)}\in T^{(1)}$. Note that if $u_{(2)}\in U^{(2)}_\mathbb{Z}$, 
$t_{(1)}\in T^{(1)}$, and $v\in V_\mathbb{Z}$, Proposition \ref{commform} 
implies that for any $p\in\mathbb{C}$ and $q\in\mathbb{Z}$, $(t_{(1)})_p v_q 
u_{(2)}$ is an integral linear combination of terms of the form $v_r 
(t_{(1)})_s 
u_{(2)}$, and is thus in $W^{(3)}_\mathbb{Z}$. This means that
\begin{equation*}
 Y_{W^{(2)}}(v,x)u_{(2)}\in U^{(2)}_\mathbb{Z}[[x,x^{-1}]],
\end{equation*}
for any $v\in V_\Z$, $u_{(2)}\in U^{(2)}_\Z$, that is, $U^{(2)}_\mathbb{Z}$ is 
a 
$V_\mathbb{Z}$-module. Since by hypothesis 
$U^{(2)}_\mathbb{Z}$ contains $T^{(2)}$ which generates $W^{(2)}_\mathbb{Z}$, 
$U^{(2)}_\mathbb{Z}=W^{(2)}_\mathbb{Z}$ and so
\begin{equation}\label{t2submod}
 \mathcal{Y}(t_{(1)},x)w_{(2)}\in W^{(3)}_\mathbb{Z}\lbrace x\rbrace
\end{equation}
for any $t_{(1)}\in T^{(1)}$ and $w_{(2)}\in W^{(2)}_\mathbb{Z}$.

Now we define $U^{(1)}_\mathbb{Z}$ as the sublattice of $W^{(1)}_\mathbb{Z}$ 
consisting of all vectors $u_{(1)}\in W^{(1)}_\mathbb{Z}$ such that
\begin{equation*}
 \mathcal{Y}(u_{(1)},x)w_{(2)}\in W^{(3)}_\mathbb{Z}\lbrace x\rbrace.
\end{equation*}
for all $w_{(2)}\in W^{(2)}_\mathbb{Z}$. In this case, $U^{(1)}_\mathbb{Z}$ is 
a 
$V_\mathbb{Z}$-submodule of 
$W^{(1)}_\mathbb{Z}$ by the iterate formula for intertwining operators, 
(\ref{assocint}). Since $U^{(1)}_\mathbb{Z}$ contains $T^{(1)}$ by 
(\ref{t2submod}), $U^{(1)}_\mathbb{Z}=W^{(1)}_\mathbb{Z}$, and we have shown 
that
\begin{equation*}
 \mathcal{Y}(w_{(1)},x)w_{(2)}\in W^{(3)}_\mathbb{Z}\lbrace x\rbrace
\end{equation*}
for any $w_{(1)}\in W^{(1)}_\mathbb{Z}$ and $w_{(2)}\in W^{(2)}_\mathbb{Z}$, 
that is, $\mathcal{Y}$ is integral with respect to $W^{(1)}_\mathbb{Z}$, 
$W^{(2)}_\mathbb{Z}$, and $W^{(3)}_\mathbb{Z}$.
\end{proof}

\begin{rema}
 Note the role of Proposition \ref{commform} in the proof of the theorem. In particular, the commutator formula for intertwining operators (obtained from the coefficient of $x_0^{-1}$ in the Jacobi identity \eqref{intwopjac}) is not enough by itself without additional assumptions on the generating set for $W^{(1)}_\Z$.
\end{rema}

We will apply Theorem \ref{intwopgen} to classify integral intertwining 
operators among modules for affine Lie algebra and lattice vertex operator 
algebras in the following sections. In both examples, the graded $\Z$-dual of 
a 
module integral form will also play a role. In the case of modules for lattice 
vertex operator algebras in Section \ref{sec:lattintwops}, we will need to 
realize the graded $\Z$-dual explicitly using intertwining operators and 
Proposition \ref{contpropo} below.

We need to recall the symmetries among spaces of intertwining operators from 
\cite{FHL}, \cite{HL1}, and \cite{HLZ2}. First, if  $W^{(1)}$,  $W^{(2)}$ and 
$W^{(3)}$ are $V$-modules and $\mathcal{Y}$ is an intertwining operator of 
type 
$\binom{W^{(3)}}{W^{(1)}\,W^{(2)}}$, then for any $r\in\mathbb{Z}$, there is an 
intertwining operator $\Omega_r(\mathcal{Y})$ of type 
$\binom{W^{(3)}}{W^{(2)}\,W^{(1)}}$ defined by
\begin{equation}\label{intwopskewsym}
 \Omega_r(\mathcal{Y})(w_{(2)},x)w_{(1)}=e^{x 
L(-1)}\mathcal{Y}(w_{(1)},e^{(2r+1)\pi i} x)w_{(2)}
\end{equation}
for $w_{(1)}\in W^{(1)}$ and $w_{(2)}\in W^{(2)}$. Moreover, for any 
$r\in\mathbb{Z}$ there is an intertwining operator $A_r(\mathcal{Y})$ of type 
$\binom{(W^{(2)})'}{W^{(1)}\,(W^{(3)})'}$ defined by
\begin{equation}\label{intwopcontra}
 \langle A_r(\mathcal{Y})(w_{(1)},x)w_{(3)}',w_{(2)}\rangle_{W^{(2)}}=\langle 
w_{(3)}',\mathcal{Y}^o_r(w_{(1)},x)w_{(2)}\rangle_{W^{(3)}}
\end{equation}
for $w_{(1)}\in W^{(1)}$, $w_{(2)}\in W^{(2)}$, and $w_{(3)}'\in (W^{(3)})'$, 
where
\begin{equation*}
 \mathcal{Y}^o_r(w_{(1)},x)w_{(2)}=\mathcal{Y}(e^{x L(1)} e^{(2r+1)\pi i 
L(0)}(x^{-L(0)})^2 w_{(1)}, x^{-1})w_{(2)}.
\end{equation*}
Then we have (see for example Propositions 3.44 and 3.46 in \cite{HLZ2}):
\begin{propo}\label{intwopsymprop}
 For any $r\in\mathbb{Z}$, the map $\Omega_r: V^{W^{(3)}}_{W^{(1)} 
W^{(2)}}\rightarrow V^{W^{(3)}}_{W^{(2)} W^{(1)}}$ is a linear isomorphism with 
inverse $\Omega_{-r-1}$. Moreover, the map $A_r: V^{W^{(3)}}_{W^{(1)} 
W^{(2)}}\rightarrow V^{(W^{(2)})'}_{W^{(1)} (W^{(3)})'}$ is a linear 
isomorphism 
with inverse $A_{-r-1}$ for any $r\in\mathbb{Z}$. 
\end{propo}

We say that a bilinear pairing $(\cdot,\cdot)$ between $V$-modules $W^{(1)}$ 
and 
$W^{(2)}$ is \textit{invariant} if
\begin{equation*}
 (Y_{W^{(1)}}(v,x)w_{(1)}, w_{(2)})=(w_{(1)}, Y_{W^{(2)}}^o(v,x)w_{(2)})
\end{equation*}
for $v\in V$, $w_{(1)}\in W^{(1)}$, and $w_{(2)}\in W^{(2)}$. It is clear that 
if there is a nondegenerate invariant bilinear pairing between $W^{(1)}$ and 
$W^{(2)}$, then $W^{(1)}$ and $W^{(2)}$ form a contragredient pair. The next 
proposition, which is a minor generalization of Remark 2.9 in \cite{Li}, shows 
how certain intertwining operators yield nondegenerate invariant pairings 
between modules:
\begin{propo}\label{contpropo}
 Suppose $V$ is a vertex operator algebra equipped with a nondegenerate 
invariant bilinear pairing $(\cdot,\cdot)_V$, and suppose $W^{(1)}$ and 
$W^{(2)}$ 
are $V$-modules. If $\mathcal{Y}$ is an intertwining operator of type 
$\binom{V}{W^{(1)}\,W^{(2)}}$, then the bilinear pairing $(\cdot,\cdot)$ 
between 
$W^{(1)}$ and $W^{(2)}$ given by
 \begin{equation}\label{contform}
  (w_{(1)},w_{(2)})=\mathrm{Res}_x\, (\mathbf{1},\mathcal{Y}^o_0(w_{(1)}, 
e^{\pi 
i} x) e^{x L(1)} w_{(2)})_V
 \end{equation}
for $w_{(1)}\in W^{(1)}$ and $w_{(2)}\in W^{(2)}$ is invariant. Moreover, if 
$W^{(1)}$ and $W^{(2)}$ are irreducible and $\mathcal{Y}$ is non-zero, then the 
pairing is nondegenerate and $W^{(1)}$ and $W^{(2)}$ form a contragredient pair.
\end{propo}
\begin{proof}
We consider the intertwining operator $\mathcal{Y}'=\Omega_0(A_0(\mathcal{Y}))$ 
of type $\binom{(W^{(2)})'}{V\cong V'\,W^{(1)}}$. The  $L(-1)$-derivative 
property (\ref{intwopderiv}) implies that $\mathcal{Y}'(\mathbf{1},x)$ equals 
its constant term $\mathbf{1}_{-1}$. Moreover, the coefficient of $x_0^{-1}$ 
in the Jacobi identity (\ref{intwopjac}) implies that
\begin{equation*}
 Y_{(W^{(2)})'}(v,x_1)\mathbf{1}_{-1}-\mathbf{1}_{-1} Y_{W^{(1)}}(v,x_1)=0
\end{equation*}
since $Y(v,x_0)\mathbf{1}$ has no negative powers of $x_0$. Thus 
$\mathbf{1}_{-1}=\varphi_\mathcal{Y}$ is a $V$-homomorphism from 
$W^{(1)}$ to $(W^{(2)})'$ and we obtain a bilinear pairing $(\cdot,\cdot)$ 
between $W^{(1)}$ and $W^{(2)}$ given by
\begin{equation*}
 (w_{(1)},w_{(2)})=\langle\varphi_\mathcal{Y}(w_{(1)}), w_{(2)}\rangle
\end{equation*}
for $w_{(1)}\in W^{(1)}$, $w_{(2)}\in W^{(2)}$, which is invariant because 
$\varphi_\mathcal{Y}$ is a homomorphism.

To show that the invariant pairing $(\cdot,\cdot)$ is also given by 
(\ref{contform}), we calculate using the definitions of $\Omega_0$ and $A_0$ 
from (\ref{intwopskewsym}) and (\ref{intwopcontra}), and identifying $V\cong 
V'$ 
via $(\cdot,\cdot)_V$. For $w_{(1)}\in W^{(1)}$ and $w_{(2)}\in W^{(2)}$,
\begin{align*}
 (w_{(1)},w_{(2)}) & =\langle\varphi_\mathcal{Y}(w_{(1)}),w_{(2)}\rangle 
=\mathrm{Res}_x\,x^{-1} 
\langle\Omega_0(A_0(\mathcal{Y}))(\mathbf{1},x)w_{(1)},w_{(2)}\rangle\\
 & =\mathrm{Res}_x\,x^{-1}\langle e^{x L(-1)} A_0(\mathcal{Y})(w_{(1)},e^{\pi 
i} 
x)\mathbf{1},w_{(2)}\rangle\\
 & =\mathrm{Res}_x\,x^{-1}\langle A_0(\mathcal{Y})(w_{(1)},e^{\pi i} 
x)\mathbf{1}, e^{x L(1)} w_{(2)}\rangle\\
 & =\mathrm{Res}_x\,x^{-1} (\mathbf{1}, \mathcal{Y}^o_0(w_{(1)}, e^{\pi i} x) 
e^{x L(1)}w_{(2)})_V.
\end{align*}
This proves the first assertion of the proposition.

To prove the nondegeneracy of $(\cdot,\cdot)$ when $W^{(1)}$, $W^{(2)}$ are  
irreducible and $\mathcal{Y}$ is non-zero, it is enough to prove that 
$\varphi_\mathcal{Y}: W^{(1)}\rightarrow (W^{(2)})'$ is an isomorphism. Since 
$W^{(1)}$ and $W^{(2)}$ are irreducible, it suffices to prove that 
$\varphi_\mathcal{Y}$ is non-zero, that is, $\mathcal{Y}'(\mathbf{1},x)\neq 0$. 
But if $\mathcal{Y}'(\mathbf{1},x)=0$, then the creation property for vertex 
operator algebras and the 
iterate formula (\ref{assocint}) imply that for any $v\in V$,
\begin{equation*}
 \mathcal{Y}'(v,x)=\mathcal{Y}'(v_{-1}\mathbf{1},x)=0.
\end{equation*}
This is a contradiction because by Proposition \ref{intwopsymprop}, 
$\mathcal{Y}'$ is non-zero if $\mathcal{Y}$ is.
\end{proof}

\section{Affine Lie algebra vertex operator algebras, their modules, and 
intertwining operators}

In this section, we shall recall what we need from the representation theory of 
vertex operator algebras based on affine Lie algebras. We fix a 
finite-dimensional complex simple Lie algebra $\mathfrak{g}$, and let 
$\mathfrak{h}$ denote a Cartan 
subalgebra of $\mathfrak{g}$. There is a unique up to scale nondegenerate 
invariant bilinear form $\langle\cdot,\cdot\rangle$ on $\mathfrak{g}$, which is 
nondegenerate on $\mathfrak{h}$ and induces a bilinear form on 
$\mathfrak{h}^*$. 
We shall normalize the form $\langle\cdot,\cdot\rangle$ on $\mathfrak{g}$ so 
that 
\begin{equation*}
 \langle\alpha,\alpha\rangle=2
\end{equation*}
for long roots $\alpha\in\mathfrak{h}^*$. Irreducible finite-dimensional 
$\mathfrak{g}$-modules are in one-to-one correspondence with dominant integral 
weights $\lambda\in\mathfrak{h}^*$. 

We now recall the affine Lie algebra
\begin{equation*} 
\widehat{\mathfrak{g}}=\mathfrak{g}\otimes\mathbb{C}[t,t^{-1}]\oplus\mathbb{C}
\mathbf{k}
\end{equation*}
where $\mathbf{k}$ is central and all other brackets are determined by
\begin{equation*}
 [a\otimes t^m, b\otimes t^n]=[a,b]\otimes t^{m+n}+m\langle 
a,b\rangle\delta_{m+n,0}\mathbf{k}
\end{equation*}
for $a,b\in\mathfrak{g}$ and $m,n\in\mathbb{Z}$. Using the notation of 
\cite{LL} 
Chapter 6, for any finite-dimensional $\mathfrak{g}$-module $U$ and any level 
$\ell\in\mathbb{C}$, we have the generalized Verma module $\gvmu$ on which 
$\mathbf{k}$ acts as the scalar $\ell$. The generalized Verma module $\gvmu$ 
has 
a unique irreducible quotient $\imu$. The $\widehat{\mathfrak{g}}$-modules 
$\gvmu$ and $\imu$ are linearly spanned by vectors of the form
\begin{equation}\label{affinespan}
 a_1(-n_1)\cdots a_k(-n_k) u
\end{equation}
for $a_i\in\mathfrak{g}$, $n_i>0$. Here we use $a(n)$ for $a\in\mathfrak{g}$ 
and 
$n\in\mathbb{Z}$ to denote the action of $a\otimes t^n$ on a 
$\widehat{\mathfrak{g}}$-module. 

In the case $U=\mathbb{C}\mathbf{1}$, the trivial one-dimensional 
$\mathfrak{g}$-module, $\gvmu=\gvmzero$ is a vertex algebra, as is its 
irreducible quotient $\imzero$. The vertex operator map is determined by
\begin{equation}\label{liealgops}
 Y(a(-1)\mathbf{1},x)=\sum_{n\in\mathbb{Z}} a(n) x^{-n-1}
\end{equation}
for $a\in\mathfrak{g}$. When $\ell\neq -h$ where $h$ is the dual Coxeter number 
of $\mathfrak{g}$, $\gvmzero$ and $\imzero$ are also vertex operator algebras.

Let $L_\lambda$ denote the irreducible $\mathfrak{g}$-module with highest 
weight 
$\lambda$. Then for any level $\ell\in\mathbb{C}$, the irreducible 
$V_{\widehat{\mathfrak{g}}}(\ell, 0)$-modules consist of the modules $L(\ell, L_\lambda)$ 
for $\lambda$ a dominant integral weight. Moreover, suppose $\theta$ is the 
highest root of $\mathfrak{g}$; if $\ell$ is a nonnegative integer, the 
irreducible $L_{\widehat{\mathfrak{g}}}(\ell,0)$-modules are given by $L(\ell, 
L_\lambda)$ where $\lambda$ is a dominant integral weight satisfying 
$\langle\lambda,\theta\rangle\leq\ell$ (\cite{FZ}; see also \cite{LL}).

As is well known, the dual of an irreducible $\mathfrak{g}$-module $L_\lambda$ is also an  
irreducible $\mathfrak{g}$-module with the action given by
\begin{equation*}
 \langle x\cdot v',v\rangle=-\langle v',x\cdot v\rangle
\end{equation*}
for $x\in\mathfrak{g}$, $v\in L_\lambda$, and $v'\in L_{\lambda}^*$. So we have 
$L_\lambda^*\cong L_{\lambda^*}$ for some dominant integral weight $\lambda^*$. 
In fact, $\lambda^*=-w(\lambda)$ where $w$ is the element in the Weyl group of 
$\mathfrak{g}$ of maximal length (see for example \cite{Hu}). Then the 
contragredient of the
$V_{\widehat{\mathfrak{g}}}(\ell,0)$-module 
$L_{\widehat{\mathfrak{g}}}(\ell,L_\lambda)$ is isomorphic to 
$L_{\widehat{\mathfrak{g}}}(\ell, L_{\lambda^*})$ (\cite{FZ}).

From now on, we assume that $\ell$ is a nonnegative integer; suppose that 
$L_{\widehat{\mathfrak{g}}}(\ell, L_{\lambda_1})$, 
$L_{\widehat{\mathfrak{g}}}(\ell, L_{\lambda_2})$, and 
$L_{\widehat{\mathfrak{g}}}(\ell, L_{\lambda_3})$ are irreducible 
$L_{\widehat{\mathfrak{g}}}(\ell,0)$-modules, so that 
$\langle\lambda_i,\theta\rangle\leq\ell$ for $i=1,2,3$. We recall from 
\cite{FZ} 
(see also \cite{Li2}) the classification of intertwining operators of type 
$\binom{L_{\widehat{\mathfrak{g}}}(\ell, 
L_{\lambda_3})}{L_{\widehat{\mathfrak{g}}}(\ell, 
L_{\lambda_1})\,L_{\widehat{\mathfrak{g}}}(\ell, L_{\lambda_2})}$. First, let 
$A(L_{\widehat{\mathfrak{g}}}(\ell,0))$ be the Zhu's algebra of 
$L_{\widehat{\mathfrak{g}}}(\ell,0)$ (see \cite{Z} for the definition); from 
\cite{FZ}, as an associative algebra,
\begin{equation*}
 A(L_{\widehat{\mathfrak{g}}}(\ell,0))\cong U(\mathfrak{g})/\langle 
x_\theta^{\ell+1}\rangle,
\end{equation*}
where $x_\theta$ is a root vector for the longest root $\theta$ of 
$\mathfrak{g}$. We also need the 
$A(L_{\widehat{\mathfrak{g}}}(\ell,0))$-bimodule 
$A(L_{\widehat{\mathfrak{g}}}(\ell,L_{\lambda_1}))$, which from \cite{FZ} is 
given by
\begin{equation*}
 A(L_{\widehat{\mathfrak{g}}}(\ell,L_{\lambda_1}))\cong (L_{\lambda_1}\otimes 
U(\mathfrak{g}))/\langle v_{\lambda_1}\otimes 
x_\theta^{\ell-\langle\lambda_1,\theta\rangle+1}\rangle,
\end{equation*}
where $v_{\lambda_1}$ is a highest weight vector of $L_{\lambda_1}$ and 
$\langle v_{\lambda_1}\otimes 
x_\theta^{\ell-\langle\lambda_1,\theta\rangle+1}\rangle$ indicates the 
subbimodule generated by the indicated 
element. The $A(L_{\widehat{\mathfrak{g}}}(\ell,0))\cong 
U(\mathfrak{g})/\langle 
x_\theta^{\ell+1}\rangle$-bimodule 
structure on $(L_{\lambda_1}\otimes U(\mathfrak{g}))/\langle 
v_{\lambda_1}\otimes x_\theta^{\ell-\langle\lambda_1,\theta\rangle+1}\rangle$ 
is 
induced by the following $U(\mathfrak{g})$-bimodule structure on 
$L_{\lambda_1}\otimes U(\mathfrak{g})$:
\begin{equation*}
 x\cdot(v\otimes y)=(x\cdot v)\otimes y+v\otimes xy
\end{equation*}
for $x,y\in U(\mathfrak{g})$, $v\in L_{\lambda_1}$, and
\begin{equation*}
 (v\otimes y)\cdot x=v\otimes yx.
\end{equation*} 
We also recall from \cite{FZ} that the lowest weight spaces $L_{\lambda_2}$ and 
$L_{\lambda_3}$ of $L_{\widehat{\mathfrak{g}}}(\ell,L_{\lambda_2})$ and 
$L_{\widehat{\mathfrak{g}}}(\ell,L_{\lambda_3})$, respectively, are (left)
$A(L_{\widehat{\mathfrak{g}}}(\ell,0))$-modules in the natural way.

The description of the space of intertwining operators of type 
$\binom{L_{\widehat{\mathfrak{g}}}(\ell, 
L_{\lambda_3})}{L_{\widehat{\mathfrak{g}}}(\ell, 
L_{\lambda_1})\,L_{\widehat{\mathfrak{g}}}(\ell, L_{\lambda_2})}$ from 
Theorem 2.11 and Corollary 2.13 in \cite{Li2} (see also Theorems 1.5.2 and 
1.5.3 in \cite{FZ}) is as follows:
\begin{equation*} 
V^{L_{\widehat{\mathfrak{g}}}(\ell, 
L_{\lambda_3})}_{L_{\widehat{\mathfrak{g}}}(\ell, L_{\lambda_1}) 
\,L_{\widehat{\mathfrak{g}}}(\ell, 
L_{\lambda_2})}\cong\mathrm{Hom}_{A(L_{\widehat{\mathfrak{g}}}(\ell,0))}(A(L_{
\widehat{\mathfrak{g}}
}(\ell,L_{\lambda_1}))\otimes_{A(L_{\widehat{\mathfrak{g}}}(\ell,0))} 
L_{\lambda_2},L_{\lambda_3}).
\end{equation*}
This space of $A(L_{\widehat{\mathfrak{g}}}(\ell,0))$-homomorphisms can be 
described more usefully using the following lemma:
\begin{lemma}
 As (left) modules for $A(L_{\widehat{\mathfrak{g}}}(\ell,0))\cong 
U(\mathfrak{g})/\langle x_\theta^{\ell+1}\rangle$,
 \begin{equation*}  
A(L_{\widehat{\mathfrak{g}}}(\ell,L_{\lambda_1}))\otimes_{A(L_{\widehat{
\mathfrak{g}}}(\ell,0))} L_{\lambda_2}\cong (L_{\lambda_1}\otimes 
L_{\lambda_2})/W,
 \end{equation*}
with $W$ the $U(\mathfrak{g})/\langle x_\theta^{\ell+1}\rangle$-module 
generated by all vectors of the form $v_{\lambda_1}\otimes 
x_\theta^{\ell-\langle\lambda_1,\theta\rangle+1}\cdot w$ for $w\in 
L_{\lambda_2}$.
\end{lemma}
\begin{proof}
 We know that $A(L_{\widehat{\mathfrak{g}}}(\ell,\lambda_1))\cong 
(L_{\lambda_1}\otimes U(\mathfrak{g}))/W'$, where $W'$ is the subbimodule 
generated by $v_{\lambda_1}\otimes 
x_\theta^{\ell-\langle\lambda_1,\theta\rangle+1}$. We first define a map
 \begin{equation*}
  \Phi: (L_{\lambda_1}\otimes U(\mathfrak{g}))\otimes_{U(\mathfrak{g})} 
L_{\lambda_2}\rightarrow (L_{\lambda_1}\otimes 
L_{\lambda_2})/W
 \end{equation*}
as follows: for $u= v\otimes x$ where $v\in L_{\lambda_1}$ and $x\in 
U(\mathfrak{g})$, and for $w\in L_{\lambda_2}$, we define
\begin{equation*}
 \Phi(u\otimes w)=v\otimes x\cdot w+W.
\end{equation*}
By the left $U(\mathfrak{g})$-module structure on $(L_{\lambda_1}\otimes 
U(\mathfrak{g}))\otimes_{U(\mathfrak{g})} L_{\lambda_2}$ and the tensor product 
$\mathfrak{g}$-module structure on $(L_{\lambda_1}\otimes L_{\lambda_2})/W$, it 
is easy to see that $\Phi$ is a $U(\mathfrak{g})$-homomorphism. Moreover, 
$\Phi$ 
induces a $U(\mathfrak{g})/\langle x_\theta^{\ell+1}\rangle$-module 
homomorphism
\begin{equation*}
 \varphi: (L_{\lambda_1}\otimes 
U(\mathfrak{g}))/W'\otimes_{U(\mathfrak{g})/\langle x_\theta^{\ell+1}\rangle} 
L_{\lambda_2}\rightarrow (L_{\lambda_1}\otimes 
L_{\lambda_2})/W
\end{equation*}
because for $y,y'\in U(\mathfrak{g})$ and $w\in L_{\lambda_2}$,
\begin{align*}
  \Phi(y  \cdot( & v_{\lambda_1}\otimes 
x_{\theta}^{\ell-\langle\lambda_1,\theta\rangle+1})\cdot y'\otimes w)   \\
 & =\Phi((y\cdot v_{\lambda_1})\otimes 
x_\theta^{\ell-\langle\lambda_1,\theta\rangle+1} y'\otimes 
w+v_{\lambda_1}\otimes 
yx_\theta^{\ell-\langle\lambda_1,\theta\rangle+1}y'\otimes w)\\
 & =(y\cdot v_{\lambda_1})\otimes 
(x_{\theta}^{\ell-\langle\lambda_1,\theta\rangle+1} y')\cdot 
w+v_{\lambda_1}\otimes (yx_{\theta}^{\ell-\langle\lambda_1,\theta\rangle+1} 
y')\cdot w+W \\
 & =y\cdot(v_{\lambda_1}\otimes 
(x_{\theta}^{\ell-\langle\lambda_1,\theta\rangle+1} y')\cdot w)+W =0.
\end{align*}

To obtain an inverse homomorphism, we define a $U(\mathfrak{g})$-homomorphism
\begin{equation*}
 \Psi: L_{\lambda_1}\otimes L_{\lambda_2}\rightarrow (L_{\lambda_1}\otimes 
U(\mathfrak{g}))/W'\otimes_{U(\mathfrak{g})/\langle x_\theta^{\ell+1}\rangle} 
L_{\lambda_2}
\end{equation*}
by defining for $u\in L_{\lambda_1}$ and $w\in L_{\lambda_2}$:
\begin{equation*}
 \Psi(u\otimes w)=(u\otimes 1+W')\otimes w.
\end{equation*}
The map $\Psi$ induces a $U(\mathfrak{g})/\langle 
x_\theta^{\ell+1}\rangle$-homomorphism
\begin{equation*}
 \psi: (L_{\lambda_1}\otimes L_{\lambda_2})/W\rightarrow (L_{\lambda_1}\otimes 
U(\mathfrak{g}))/W'\otimes_{U(\mathfrak{g})/\langle x_\theta^{\ell+1}\rangle} 
L_{\lambda_2}
\end{equation*}
because for $x\in U(\mathfrak{g})$ and $w\in L_{\lambda_2}$,
\begin{align*}
  \Psi(x\cdot( & v_{\lambda_1}\otimes 
x_\theta^{\ell-\langle\lambda_1,\theta\rangle+1}\cdot w)) 
 = \Psi((x\cdot v_{\lambda_1})\otimes 
x_\theta^{\ell-\langle\lambda_1,\theta\rangle+1}\cdot w +v_{\lambda_1}\otimes 
(xx_\theta^{\ell-\langle\lambda_1,\theta\rangle+1})\cdot w)\\
 & = ((x\cdot v_{\lambda_1})\otimes 1 +W')\otimes 
x_\theta^{\ell-\langle\lambda_1,\theta\rangle+1}\cdot w +(v_{\lambda_1}\otimes 
1 
+W')\otimes xx_\theta^{\ell-\langle\lambda_1,\theta\rangle+1}\cdot w\\
 & =((x\cdot v_{\lambda_1})\otimes 
x_\theta^{\ell-\langle\lambda_1,\theta\rangle+1}+W')\otimes 
w+(v_{\lambda_1}\otimes 
xx_\theta^{\ell-\langle\lambda_1,\theta\rangle+1}+W')\otimes w\\
 & = (x\cdot(v_{\lambda_1}\otimes 
x_\theta^{\ell-\langle\lambda_1,\theta\rangle+1})+W')\otimes w=0.
\end{align*}
From the definitions, it is easy to see that $\varphi$ and $\psi$ are inverses 
of each other, so they give $A(L_{\widehat{\mathfrak{g}}}(\ell,0))\cong 
U(\mathfrak{g})/\langle x_\theta^{\ell+1}\rangle$-module isomorphisms.
\end{proof}
From this lemma, we have (see Theorem 3.2.3 in \cite{FZ}):
\begin{equation*}
 V^{L_{\widehat{\mathfrak{g}}}(\ell, 
L_{\lambda_3})}_{L_{\widehat{\mathfrak{g}}}(\ell, L_{\lambda_1}) 
\,L_{\widehat{\mathfrak{g}}}(\ell, 
L_{\lambda_2})}\cong\mathrm{Hom}_{\mathfrak{g}}((L_{\lambda_1}\otimes 
L_{\lambda_2})/W, 
L_{\lambda_3}).
\end{equation*}
From \cite{Li2} and \cite{FZ}, we can describe the isomorphism as follows. An 
intertwining operator $\mathcal{Y}$ of type 
$\binom{L_{\widehat{\mathfrak{g}}}(\ell, 
L_{\lambda_3})}{L_{\widehat{\mathfrak{g}}}(\ell, 
L_{\lambda_1})\,L_{\widehat{\mathfrak{g}}}(\ell, L_{\lambda_2})}$ induces a 
$\mathfrak{g}$-homomorphism
\begin{equation*}
 \pi(\mathcal{Y}): L_{\lambda_1}\otimes L_{\lambda_2}\rightarrow L_{\lambda_3}
\end{equation*}
given by
\begin{equation}\label{intwopiso}
 \pi(\mathcal{Y})(w_{(1)}\otimes w_{(2)})=w_{(1)}(0) w_{(2)}
\end{equation}
using the notation of (\ref{homogeneousops}), where $w_{(1)}\in 
L_{\lambda_1}\subseteq L_{\widehat{\mathfrak{g}}}(\ell, L_{\lambda_1})$ and 
$w_{(2)}\in L_{\lambda_2}\subseteq L_{\widehat{\mathfrak{g}}}(\ell, 
L_{\lambda_2})$. This homomorphism $\pi(\mathcal{Y})$ must equal $0$ on $W$, so it 
induces a homomorphism, which we also call $\pi(\mathcal{Y})$, from  
$(L_{\lambda_1}\otimes L_{\lambda_2})/W$ to $L_{\lambda_3}$. Note that it is 
not 
trivial to show that $\mathcal{Y}\mapsto\pi(\mathcal{Y})$ is a linear 
isomorphism; the most difficult part is to show that a 
$\mathfrak{g}$-homomorphism
\begin{equation*}
 f: (L_{\lambda_1}\otimes L_{\lambda_2})/W\rightarrow L_{\lambda_3}
\end{equation*}
extends to a (unique) intertwining operator of type 
$\binom{L_{\widehat{\mathfrak{g}}}(\ell, 
L_{\lambda_3})}{L_{\widehat{\mathfrak{g}}}(\ell, 
L_{\lambda_1})\,L_{\widehat{\mathfrak{g}}}(\ell, L_{\lambda_2})}$.

\section{Integral intertwining operators among modules for affine Lie algebra 
vertex operator algebras}

In this section we fix a nonnegative integral level $\ell$. We will classify 
integral intertwining operators among $\imzero$-modules, using integral forms 
which were first introduced in a Lie-algebraic setting in \cite{G1}. We shall 
now recall these integral forms; see \cite{M2} for more details.

The universal enveloping algebra $U(\widehat{\mathfrak{g}})$ of 
$\widehat{\mathfrak{g}}$ has an integral form 
$U_{\mathbb{Z}}(\widehat{\mathfrak{g}})$: it is the subring of 
$U(\widehat{\mathfrak{g}})$ generated by the vectors $(x_{\alpha}\otimes 
t^n)^k/k!$ for $k\geq 0$, where 
$n\in\mathbb{Z}$ and $x_\alpha $ is a root vector in a Chevalley basis of 
$\mathfrak{g}$ (\cite{G1}, see also \cite{Mi} and \cite{P}). We describe a 
basis 
of $U_\mathbb{Z}(\widehat{\mathfrak{g}})$ given in \cite{Mi}. Consider a 
Chevalley basis $\lbrace x_\alpha, h_i\rbrace$ of $\mathfrak{g}$, where 
$x_\alpha$ is a root vector corresponding to the root $\alpha$ and $h_i$ is the 
coroot corresponding to a simple root $\alpha_i$; suppose that $\theta$ is the 
highest root of $\mathfrak{g}$ with respect to the simple roots 
$\lbrace\alpha_i\rbrace$. Then $\widehat{\mathfrak{g}}$ has an 
integral basis consisting of the vectors
\begin{equation}\label{affinebasis}
 x_\alpha\otimes t^n,\,\,\,h_i\otimes t^n,\,\,\,-h_\theta\otimes t^0+\mathbf{k}
\end{equation}
where $n\in\mathbb{Z}$. Given any order of this basis, a basis for 
$U_\mathbb{Z}(\widehat{\mathfrak{g}})$ consists of ordered products of elements 
of the following forms:
\begin{equation}\label{ubasis1}
 \dfrac{(x_\alpha\otimes t^n)^m}{m!},\,\,\,\binom{h_i\otimes 
t^0+m-1}{m},\,\,\,\binom{-h_\theta\otimes t^0+\mathbf{k}+m-1}{m}
\end{equation}
where $n\in\mathbb{Z}$ and $m\geq 0$, as well as coefficients of powers of $x$ 
in 
series of the form
\begin{equation}\label{ubasis2}
 \mathrm{exp}\left(\sum_{j\geq 1} (h_i\otimes t^{nj})\dfrac{x^j}{j}\right)
\end{equation}
for $n\in\mathbb{Z}\setminus\lbrace 0\rbrace$.

Suppose $L_\lambda$ is a finite-dimensional irreducible $\mathfrak{g}$-module 
where $\lambda$ is a dominant integral weight of $\mathfrak{g}$, and suppose 
$v_\lambda$ is a highest weight vector of $L_\lambda$. Then $\gvmlambda$ and 
$\imlambda$ have integral forms $\gvmlambda_\mathbb{Z}$ and 
$\imlambda_\mathbb{Z}$ given by $U_\mathbb{Z}(\widehat{\mathfrak{g}})\cdot 
v_\lambda$. Given that $L_\lambda$ is included in $\gvmlambda$ and $\imlambda$ 
as their lowest conformal weight spaces (recall \eqref{affinespan}), we can 
define
\begin{equation*}
 (L_\lambda)_\mathbb{Z}=L_\lambda\cap 
U_{\mathbb{Z}}(\widehat{\mathfrak{g}})\cdot v_\lambda,
\end{equation*}
a sublattice of both $\gvmlambda_\mathbb{Z}$ and $\imlambda_\mathbb{Z}$ which, 
by \eqref{ubasis1}, agrees with the classical integral form of a 
$\mathfrak{g}$-module constructed in, for example, \cite{Hu} Chapter 7. In 
light of the basis for $U_\mathbb{Z}(\widehat{\mathfrak{g}})$ described above, 
we have
\begin{propo}\label{affintzformbasis}
 The integral forms  $V_{\widehat{\mathfrak{g}}}(\ell, L_\lambda)_\mathbb{Z}$ 
and 
$L_{\widehat{\mathfrak{g}}}(\ell, L_\lambda)_\mathbb{Z}$ are spanned by vectors 
of the form $P\cdot u$, where $u\in (L_\lambda)_\mathbb{Z}$ and $P$ is a 
product 
of operators of the form $x_\alpha(-n)^m/m!$, for $\alpha$ a root of 
$\mathfrak{g}$ and $n,m> 0$, and coefficients of powers of $x$ in series of the 
form
 \begin{equation*}
   \mathrm{exp}\left(\sum_{j\geq 1} h_i(-nj)\dfrac{x^j}{j}\right)
 \end{equation*}
for $n>0$.
\end{propo}
\begin{proof}
 This follows immediately from (\ref{ubasis1}), (\ref{ubasis2}), and an 
appropriate choice of order on the basis (\ref{affinebasis}).
\end{proof}

The Lie-algebraic integral forms $\gvmlambda_\mathbb{Z}$ and 
$\imlambda_\mathbb{Z}$ are also vertex algebraic integral forms, as shown by 
the 
following two results proved in \cite{M2}:
\begin{theo}\label{affalgzform}
 The integral form 
$V_{\widehat{\mathfrak{g}}}(\ell,0)_\mathbb{Z}$ is
the vertex subring of $V_{\widehat{\mathfrak{g}}}(\ell,0)$ generated by the
vectors $\frac{x_\alpha(-1)^k}{k!}\mathbf{1}$ where $k\geq 0$ and $x_\alpha$ is
the root vector corresponding to the root $\alpha$ in the chosen Chevalley basis
of $\mathfrak{g}$. Moreover, if $L_\lambda$ is a finite-dimensional 
$\mathfrak{g}$-module with highest weight vector $v_\lambda$ and $W$ is 
$V_{\widehat{\mathfrak{g}}}(\ell,L_\lambda)$ or 
$L_{\widehat{\mathfrak{g}}}(\ell,L_\lambda)$, 
then $W_\mathbb{Z}$ is the 
 $V_{\widehat{\mathfrak{g}}}(\ell,0)_\mathbb{Z}$-module generated by 
$v_\lambda$.
\end{theo}
\begin{corol}\label{affalgzformcorol}
 The integral form $L_{\widehat{\mathfrak{g}}}(\ell,0)_\mathbb{Z}$ of 
$L_{\widehat{\mathfrak{g}}}(\ell,0)$ is the integral form of 
$L_{\widehat{\mathfrak{g}}}(\ell,0)$ as a vertex algebra generated by the 
vectors $\frac{x_\alpha(-1)^k}{k!}\mathbf{1}$ where $\alpha$ is a root and 
$k\geq 0$.
\end{corol}

It was also shown in \cite{M2} that the graded $\mathbb{Z}$-dual of a 
$\gvmzero_\mathbb{Z}$- or $\imzero_\mathbb{Z}$-module is also a 
$\gvmzero_\Z$- or $\imzero_\Z$-module. In particular, suppose that $\lambda$ 
is a dominant integral weight of $\mathfrak{g}$ satisfying 
$\langle\lambda,\theta\rangle\leq\ell$ where $\theta$ is the highest root of 
$\mathfrak{g}$, so that $\imlambda$ is an $\imzero$-module. Then the graded 
$\Z$-dual of $\imlambda_\Z$ is the $\imzero_\Z$-module $\imlambda_\Z'$, which 
is an integral form of $L_{\widehat{\mathfrak{g}}}(\ell,L_{\lambda^*})$ that is 
generally different from $L_{\widehat{\mathfrak{g}}}(\ell,L_{\lambda^*})_\Z.$ 

\begin{rema}
 The precise identity of the integral form $\imlambda_\Z'\subseteq L_{\ghat}(\ell,L_{\lambda^*})$ depends on the choice of isomorphism $\imlambda'\cong L_{\ghat}(\ell, L_{\lambda^*})$. Because $\imlambda$ is irreducible, such an isomorphism amounts to a $\mathfrak{g}$-module isomorphism $L_\lambda^*\cong L_{\lambda^*}$ , in which the $\Z$-dual of $(L_\lambda)_\Z$ is identified with a lattice $(L_\lambda)_\Z'\subseteq L_{\lambda^*}$.
\end{rema}

 We now come to our main theorem on integral intertwining operators among 
$\imzero$-modules. We consider intertwining operators among three 
$L_{\widehat{\mathfrak{g}}}(\ell,0)$-modules 
$L_{\widehat{\mathfrak{g}}}(\ell,L_{\lambda_i})$ for $i=1,2,3$. In the 
statement 
of the following theorem, we use 
$U_\mathbb{Z}(\mathfrak{g})$ to denote 
$U_\mathbb{Z}(\widehat{\mathfrak{g}})\cap 
U(\mathfrak{g})$, which by (\ref{ubasis1}) has as a basis ordered monomials in 
elements of the forms
\begin{equation*}
 \dfrac{x_\alpha^m}{m!},\,\,\,\binom{h_i+m-1}{m}
\end{equation*}
where $m\geq 0$ and we identify $x_\alpha=x_\alpha\otimes t^0$ and 
$h_i=h_i\otimes t^0$.
\begin{theo}
 The lattice of intertwining operators within 
$V^{L_{\widehat{\mathfrak{g}}}(\ell,L_{\lambda_3})}_{L_{\widehat{\mathfrak{g}}}
(\ell,L_{\lambda_1})\,L_{\widehat{\mathfrak{g}}}(\ell,L_{\lambda_2})}$ which 
are 
integral with respect to 
$L_{\widehat{\mathfrak{g}}}(\ell,L_{\lambda_1})_\mathbb{Z}$, 
$L_{\widehat{\mathfrak{g}}}(\ell,L_{\lambda_2})_\mathbb{Z}$, and 
$L_{\widehat{\mathfrak{g}}}(\ell,L_{\lambda_3^*})_\mathbb{Z}'$ is isomorphic to
 \begin{equation*}
\mathrm{Hom}_{U_\mathbb{Z}(\mathfrak{g})}(((L_{\lambda_1})_\mathbb{Z}\otimes_{
\mathbb{Z}}(L_{\lambda_2})_\mathbb{Z})/W_\mathbb{Z}, 
(L_{\lambda_3^*})_\mathbb{Z}')
 \end{equation*}
where $W_{\mathbb{Z}}$ is the $U_\mathbb{Z}(\mathfrak{g})$-submodule of 
$(L_{\lambda_1})_\mathbb{Z}\otimes_{\mathbb{Z}}(L_{\lambda_2})_\mathbb{Z}$ 
generated by vectors of the form
\begin{equation*}
v_{\lambda_1}\otimes\dfrac{x_{\theta}^{\ell-\langle\lambda_1,\theta\rangle+1}}{
(\ell-\langle\lambda_1,\theta\rangle+1)!}\cdot w.
\end{equation*}
Here $v_{\lambda_1}$ is a highest weight vector generating 
$(L_{\lambda_1})_\mathbb{Z}$ as a $U_{\mathbb{Z}}(\mathfrak{g})$-module, and 
$w\in (L_{\lambda_2})_\mathbb{Z}$.
\end{theo}

\begin{rema}
 Note that $(L_{\lambda_3^*})'_\Z$ is an integral form of $L_{\lambda_3}$, the 
lowest conformal weight space of $L_{\widehat{\mathfrak{g}}}(\ell, 
L_{\lambda_3})$.
\end{rema}

\begin{rema}
 We use vectors of the form 
$v_{\lambda_1}\otimes\frac{x_{\theta}^{\ell-\langle\lambda_1,\theta\rangle+1}}{
(\ell-\langle\lambda_1,\theta\rangle+1)!}\cdot w$ to generate $W_\mathbb{Z}$, 
rather than vectors of the form $v_{\lambda_1}\otimes 
x_{\theta}^{\ell-\langle\lambda_1,\theta\rangle+1}\cdot w$, to avoid 
unnecessary 
torsion in the quotient $((L_{\lambda_1})_\mathbb{Z}\otimes_\mathbb{Z} 
(L_{\lambda_2})_\mathbb{Z})/W_\mathbb{Z}$.
\end{rema}

\begin{rema}
 Note that this theorem determines which intertwining operators of type 
$\binom{L_{\widehat{\mathfrak{g}}}(\ell,L_{\lambda_3})}{L_{\widehat{\mathfrak{g}}}
(\ell,L_{\lambda_1})\,L_{\widehat{\mathfrak{g}}}(\ell,L_{\lambda_2})}$ satisfy
\begin{equation*}
 \langle w_{(3)}',\mathcal{Y}(w_{(1)},x)w_{(2)}\rangle\in\mathbb{Z}\lbrace 
x\rbrace
\end{equation*}
for any $w_{(1)}\in L_{\widehat{\mathfrak{g}}}(\ell, 
L_{\lambda_1})_\mathbb{Z}$, 
$w_{(2)}\in L_{\widehat{\mathfrak{g}}}(\ell, L_{\lambda_2})_\mathbb{Z}$, and 
$w_{(3)}'\in L_{\widehat{\mathfrak{g}}}(\ell, L_{\lambda_3^*})_\mathbb{Z}$, 
rather than which intertwining operators satisfy
\begin{equation*}
 \mathcal{Y}: L_{\widehat{\mathfrak{g}}}(\ell,L_{\lambda_1})_\mathbb{Z}\otimes 
L_{\widehat{\mathfrak{g}}}(\ell,L_{\lambda_2})_\mathbb{Z}\rightarrow 
L_{\widehat{\mathfrak{g}}}(\ell,L_{\lambda_3})_\mathbb{Z}\lbrace x\rbrace.
\end{equation*}
It is not clear that there will generally be any non-zero intertwining 
operators 
which satisfy the latter condition.
\end{rema}

\begin{proof}
 First suppose an intertwining operator $\mathcal{Y}$ is integral with respect 
to $L_{\widehat{\mathfrak{g}}}(\ell,L_{\lambda_1})_\mathbb{Z}$, 
$L_{\widehat{\mathfrak{g}}}(\ell,L_{\lambda_2})_\mathbb{Z}$, and 
$L_{\widehat{\mathfrak{g}}}(\ell,L_{\lambda_3^*})_\mathbb{Z}'$. Then the map
 \begin{equation*}
  \pi(\mathcal{Y}): w_{(1)}\otimes w_{(2)}\mapsto w_{(1)}(0)w_{(2)}
 \end{equation*}
given in (\ref{intwopiso}) sends 
$(L_{\lambda_1})_\mathbb{Z}\otimes_{\mathbb{Z}}(L_{\lambda_2})_\mathbb{Z}$ into 
$(L_{\lambda_3^*})_\mathbb{Z}'$ and equals zero on $W_\mathbb{Z}\subseteq 
W=\langle v_{\lambda_1}\otimes 
x_{\theta}^{\ell-\langle\lambda_1,\theta\rangle+1}\cdot w\rangle$.

Conversely, suppose
\begin{equation*}
 f: 
(L_{\lambda_1})_\mathbb{Z}\otimes_{\mathbb{Z}}(L_{\lambda_2})_\mathbb{Z}
\rightarrow (L_{\lambda_3^*})_\mathbb{Z}'
\end{equation*}
is a $U_\mathbb{Z}(\mathfrak{g})$-homomorphism which equals zero on 
$W_\mathbb{Z}$. Then $f$ extends to a $U(\mathfrak{g})$-homomorphism from 
$L_{\lambda_1}\otimes L_{\lambda_2}$ to $L_{\lambda_3}$ 
which equals zero on $W$. Thus since $\pi$ given in (\ref{intwopiso}) is an 
isomorphism, $f$ induces a unique intertwining operator
\begin{equation*}
 \mathcal{Y}_f: L_{\widehat{\mathfrak{g}}}(\ell,L_{\lambda_1})\otimes  
L_{\widehat{\mathfrak{g}}}(\ell,L_{\lambda_2})\rightarrow  
L_{\widehat{\mathfrak{g}}}(\ell,L_{\lambda_3})\lbrace x\rbrace
\end{equation*}
which satisfies $\pi(\mathcal{Y}_f)=f$ and 
\begin{equation}\label{intwopintegrality}
 \langle w_{(3)}', \mathcal{Y}_f(w_{(1)},x)w_{(2)}\rangle\in\mathbb{Z}\lbrace 
x\rbrace
\end{equation}
for $w_{(1)}\in (L_{\lambda_1})_{\mathbb{Z}}$, $w_{(2)}\in 
(L_{\lambda_2})_{\mathbb{Z}}$, and $w_{(3)}'\in (L_{\lambda_3^*})_\mathbb{Z}$. 
To complete the proof, it is enough to show that \eqref{intwopintegrality} 
holds 
for all $w_{(1)}\in L_{\widehat{\mathfrak{g}}}(\ell,L_{\lambda_1})_\mathbb{Z}$, 
$w_{(2)}\in L_{\widehat{\mathfrak{g}}}(\ell,L_{\lambda_2})_\mathbb{Z}$, and 
$w_{(3)}'\in L_{\widehat{\mathfrak{g}}}(\ell,L_{\lambda_3^*})_\mathbb{Z}$.

Since $(L_{\lambda_1})_\mathbb{Z}$ and $(L_{\lambda_2})_\mathbb{Z}$ generate $ 
L_{\widehat{\mathfrak{g}}}(\ell,L_{\lambda_1})_\mathbb{Z}$ and  
$L_{\widehat{\mathfrak{g}}}(\ell,L_{\lambda_2})$, respectively, as $ 
L_{\widehat{\mathfrak{g}}}(\ell,0)_\mathbb{Z}$-modules, by Theorem 
\ref{intwopgen} it is enough to show that (\ref{intwopintegrality}) holds for 
$w_{(1)}\in (L_{\lambda_1})_{\mathbb{Z}}$, $w_{(2)}\in 
(L_{\lambda_2})_{\mathbb{Z}}$, and $w_{(3)}'\in 
L_{\widehat{\mathfrak{g}}}(\ell,L_{\lambda_3^*})_\mathbb{Z}$. Since 
(\ref{intwopintegrality}) holds when $w_{(3)}'\in 
(L_{\lambda_3^*})_\mathbb{Z}$, 
by Proposition \ref{affintzformbasis} it is enough to show that if it holds for 
some particular $w_{(3)}'\in 
L_{\widehat{\mathfrak{g}}}(\ell,L_{\lambda_3^*})_\mathbb{Z}$, then it also 
holds 
for the coefficients of powers of $y$ in
\begin{equation*}
 \mathrm{exp}\left( x_\alpha(-n) y\right)\cdot w_{(3)}'
\end{equation*}
where $\alpha$ is a root and $n>0$, and for coefficients of powers of $y$ in
\begin{equation*}
 \mathrm{exp}\left(\sum_{j\geq 1} h_i(-nj)\dfrac{y^j}{j}\right)\cdot w_{(3)}'
\end{equation*}
where $n>0$.

We will need to use the following commutator formula which holds for any 
$a\in\mathfrak{g}$, $m\in\mathbb{Z}$, and intertwining operator $\mathcal{Y}$, 
which follows from the Jacobi identity 
(\ref{intwopjac}) by setting $v=a(-1)\mathbf{1}$ and taking the coefficient of 
$x_0^{-1} x_1^{-m-1}$: for $w_{(1)}\in L_{\lambda_1}$,
\begin{equation*}
 [a(m),\mathcal{Y}(w_{(1)},x)]=x^m\mathcal{Y}(a(0)w_{(1)},x).
\end{equation*}
Then if also $w_{(2)}\in L_{\lambda_2}$ and $m>0$,
\begin{equation}\label{intwopcomm}
 a(m)\mathcal{Y}(w_{(1)},x)w_{(2)}=x^m\mathcal{Y}(a(0)w_{(1)},x)w_{(2)}.
\end{equation}
We will also use the fact that for any 
$L_{\widehat{\mathfrak{g}}}(\ell,0)$-module $W$, $w\in W$, $w'\in W'$, 
$a\in\mathfrak{g}$, and $n\in\mathbb{Z}$,
\begin{equation*}
 \langle a(n)w',w\rangle =\langle w', -a(-n)w\rangle.
\end{equation*}
 This follows from (\ref{oppopdef}) and (\ref{liealgops}).

Now suppose (\ref{intwopintegrality}) holds for $w_{(3)}'\in 
L_{\widehat{\mathfrak{g}}}(\ell, L_{\lambda_3^*})_\mathbb{Z}$ and $n>0$; then 
by 
(\ref{intwopcomm}),
\begin{align*}
 \langle\mathrm{exp}(x_\alpha(-n)y)\cdot w_{(3)}', & \,
\mathcal{Y}_f(w_{(1)},x)w_{(2)}\rangle  =\langle 
w_{(3)}',\mathrm{exp}(-x_\alpha(n)y)\mathcal{Y}_f (w_{(1)},x)w_{(2)}\rangle\\
 &=\langle w_{(3)}',\mathcal{Y}_f(\mathrm{exp}(-x_\alpha(0) x^n y)\cdot 
w_{(1)},x)w_{(2)}\in\mathbb{Z}\lbrace x,y\rbrace
\end{align*}
because for any $m\geq 0$, $\frac{x_\alpha(0)^m}{m!}\cdot w_{(1)}\in 
(L_{\lambda_1})_\mathbb{Z}$. Additionally, for each $n>0$,
\begin{align*}
 & \left\langle\mathrm{exp}  \left(\sum_{j\geq 1} 
h_i(-nj)\dfrac{y^j}{j}\right)\cdot 
w_{(3)}',\,\mathcal{Y}_f(w_{(1)},x)w_{(2)}\right\rangle \\
 &\hspace{3em} =\left\langle w_{(3)}',\mathrm{exp}\left(\sum_{j\geq 1} 
-h_i(nj)\dfrac{y^j}{j}\right)\mathcal{Y}_f(w_{(1)},x)w_{(2)}\right\rangle\\
 &\hspace{3em} =\left\langle w_{(3)}',\,\mathcal{Y}_f 
(\mathrm{exp}\left(-h_i(0)\sum_{j\geq 1} 
(-1)^j\dfrac{(-x^n y)^j}{j}\right)\cdot w_{(1)},x)w_{(2)}\right\rangle \\
 &\hspace{3em} =\langle 
w_{(3)}',\,\mathcal{Y}_f(\mathrm{exp}(h_i(0)\mathrm{log}(1-x^n y))\cdot 
w_{(1)},x)w_{(2)}\rangle\\
 &\hspace{3em} =\langle w_{(3)}',\,\mathcal{Y}_f((1-x^n y)^{h_i(0)}\cdot 
w_{(1)},x)w_{(2)}\rangle\in\mathbb{Z}\lbrace x,y\rbrace
\end{align*}
because
\begin{equation*}
 (1-x^n y)^{h_i(0)}=\sum_{k\geq 0}(-1)^k\binom{h_i(0)}{k} (x^n y)^k.
\end{equation*}
Note that $\binom{h_i(0)}{k}\cdot w_{(1)}\in (L_{\lambda_1})_\mathbb{Z}$ 
because 
$\lambda_1$ is a dominant integral weight of $\mathfrak{g}$, and thus $h_i(0)$ 
acts on basis elements of $(L_{\lambda_1})_\mathbb{Z}$ as integers. This 
completes the proof.
\end{proof}

\section{Lattice vertex operator algebras, their modules, and intertwining 
operators}

In this section we recall the construction of vertex operator algebras from 
even 
lattices, their modules, and intertwining operators among their modules; see 
\cite{FLM}, \cite{DL}, and \cite{LL} for more details. We fix a nondegenerate 
even lattice $L$ with symmetric bilinear form 
$\left\langle\cdot,\cdot\right\rangle $, and then set 
$\mathfrak{h}=\mathbb{C}\otimes_\mathbb{Z} L$, an abelian Lie algebra. Thus we 
can form the Heisenberg vertex operator algebra 
$M(1)=V_{\widehat{\mathfrak{h}}}(1,0)$.

We also need the twisted group algebra of $L^\circ$, the dual lattice of $L$, 
which is constructed from a central extension of $L^\circ$ by a finite cyclic 
group. For our purposes here, we may view the twisted group algebra 
$\mathbb{C}\lbrace L^\circ\rbrace$  as the associative algebra with basis 
vectors $e_\beta$ for $\beta\in L^\circ$ with identity $e_0$ and product
\begin{equation*}
 e_\beta e_\gamma = \varepsilon(\beta,\gamma) e_{\beta+\gamma}.
\end{equation*}
Here $\varepsilon$ maps each pair of lattice elements to a root of unity and 
may 
be chosen to satisfy
\begin{equation*} 
\varepsilon(\alpha+\beta,\gamma)=\varepsilon(\alpha,\gamma)\varepsilon(\beta,
\gamma),\;\;\;\varepsilon(\alpha,\beta+\gamma)=\varepsilon(\alpha,
\beta)\varepsilon(\alpha,\gamma)
\end{equation*}
for all $\alpha,\beta,\gamma\in L^\circ$. We will also use the commutator map 
$c$ of our central extension, which is defined by
\begin{equation*}
 e_\beta e_\gamma = c(\beta,\gamma) e_\gamma e_\beta
\end{equation*}
for $\beta,\gamma\in L^\circ$. Note that from the definitions,
\begin{equation}\label{commcocycle}
 c(\beta,\gamma)=\varepsilon(\beta,\gamma)\varepsilon(\gamma,\beta)^{-1}
\end{equation}
for all $\beta,\gamma\in L^\circ$. The commutator map $c$ is required to satisfy
\begin{equation*}
 c(\alpha,\beta)=(-1)^{\langle\alpha,\beta\rangle}
\end{equation*}
for $\alpha,\beta\in L$. (Note that $L\subseteq L^\circ$ since $L$ is even, and 
thus integral.)

We define an action of $\mathfrak{h}=\mathfrak{h}\otimes t^0$ on 
$\mathbb{C}\lbrace L^\circ\rbrace$ by
\begin{equation*}
 h(0) e_\beta =\langle h,\beta\rangle e_\beta
\end{equation*}
for $h\in\mathfrak{h}$, $\beta\in L^\circ$. Also for $\beta\in L^\circ$ and $x$ 
a formal variable, we define a map $x^\beta$ on $\mathbb{C}\lbrace 
L^\circ\rbrace$ by
\begin{equation*}
 x^\beta e_\gamma = e_\gamma x^{\langle\beta,\gamma\rangle}
\end{equation*}
for any $\gamma\in L^\circ$.

For any subset $S\subseteq L^\circ$, we use $\mathbb{C}\lbrace S\rbrace$ to 
denote the subspace of $\mathbb{C}\lbrace L^\circ\rbrace$ spanned by the 
vectors 
$e_\beta$ for $\beta\in S$. Then as a vector space, the vertex operator algebra 
associated to $L$ is 
\begin{equation*}
 V_L=M(1)\otimes\mathbb{C}\lbrace L\rbrace.
\end{equation*}
For $\alpha\in L$, we use the notation $\iota(e_\alpha)$ to denote the element 
$1\otimes e_\alpha\in V_L$. The vacuum of $V_L$ is $\mathbf{1}=\iota(e_0)$, and 
the vertex algebra structure is generated by the vectors $\iota(e_\alpha)$ for 
$\alpha\in L$, whose vertex operators are
\begin{equation}\label{lattops}
 Y(\iota(e_\alpha),x)=E^- (-\alpha,x)E^+ (-\alpha,x) e_\alpha x^{\alpha}
\end{equation}
where
\begin{equation*}
 E^\pm (\beta, x)=\mathrm{exp}\left( \sum_{n\in\pm\mathbb{Z}_+} 
\frac{\beta(n)}{n} x^{-n}\right) 
\end{equation*}
for any $\beta\in L^\circ$, and $e_\alpha$ denotes the left multiplication 
action on $\mathbb{C}\lbrace L\rbrace$.
\begin{rema}
 The vertex algebra $V_L$ is a vertex operator algebra in the sense of 
\cite{FLM} only when $L$ is positive definite. If $L$ is not positive definite, 
$V_L$ is a conformal vertex algebra in the sense of \cite{HLZ}: there are 
infinitely many weight spaces of negative conformal weight, and the weight 
spaces are infinite dimensional. However, all the results of this paper still 
hold in this generality.
\end{rema}

We now recall the irreducible modules for $V_L$. The space
\begin{equation*}
 V_{L^\circ}=M(1)\otimes\mathbb{C}\lbrace 
L^\circ\rbrace
\end{equation*}
is a $V_L$-module, with vertex operators for the $\iota(e_\alpha)$ acting on 
$V_{L^\circ}$ the same as in \eqref{lattops}. Moreover, the submodules
\begin{equation*}
 V_{\beta +L}=M(1)\otimes\mathbb{C}\lbrace 
\beta+L\rbrace ,
\end{equation*}
where $\beta$ runs over coset representatives of $L^\circ /L$, exhaust the 
irreducible $V_L$-modules up to equivalence.

We now recall the intertwining operators among $V_L$-modules from \cite{DL}. 
For 
any $\beta\in L^\circ$, we have an intertwining operator
\begin{equation*}
 \mathcal{Y}_\beta: V_{\beta+L}\otimes V_{L^\circ}\rightarrow 
V_{L^\circ}\lbrace 
x\rbrace,
\end{equation*}
where for any $\gamma\in\beta+L$,
\begin{equation}\label{latintwop}
 \mathcal{Y}_\beta(\iota(e_\gamma),x)=E^-(-\gamma,x)E^+(-\gamma,x) e_\gamma 
x^{\gamma} 
e^{\pi i\beta} c(\cdot,\beta).
\end{equation}
Here, the operator $e^{\pi i\beta}$ on $V_{L^\circ}$ is given by
\begin{equation*}
 e^{\pi i\beta}\cdot u\otimes\iota(e_\gamma)=e^{\pi 
i\langle\beta,\gamma\rangle} 
u\otimes\iota(e_\gamma)
\end{equation*}
for $u\in M(1)$, $\gamma\in L^\circ$, and the 
operator $c(\cdot,\beta)$ is defined by
\begin{equation*}
 c(\cdot,\beta)\cdot u\otimes\iota(e_\gamma)=c(\gamma,\beta) 
u\otimes\iota(e_\gamma).
\end{equation*}
Then for any $\gamma\in L^\circ$, $\mathcal{Y}_\beta\vert_{V_{\beta+L}\otimes 
V_{\gamma+L}}$ is an 
intertwining operator of type 
$\binom{V_{\beta+\gamma+L}}{V_{\beta+L}\,V_{\gamma+L}}$. From \cite{DL} we have:
\begin{propo}
  For any $\alpha,\beta,\gamma\in L^\circ$,
 \begin{equation*}
 V^\alpha_{\beta\,\gamma}=V^{V_{\alpha+L}}_{V_{\beta+L} 
V_{\gamma+L}}=\left\lbrace
\begin{array}{ccc}  
\mathbb{C}\mathcal{Y}_\beta\vert_{V_{\beta+L}\otimes V_{\gamma+L}} & 
\mathrm{if} 
& \alpha=\beta+\gamma\\ 
0 & \mathrm{if} & \alpha\neq\beta+\gamma\\ 
\end{array}.\right.
\end{equation*}
\end{propo}
\begin{rema}
 Note that since for any $\beta\in L^\circ$, 
$\mathcal{Y}_\beta\vert_{V_{\beta+L}\otimes V_{-\beta+L}}$ is a non-zero 
intertwining operator of type $\binom{V_L}{V_{\beta+L}\,V_{-\beta+L}}$, 
Proposition \ref{contpropo} implies that $V_{\beta+L}'\cong V_{-\beta+L}$.
\end{rema}

% The vertex operator algebra $V_L$ has a unique up to scale nondegenerate 
% invariant bilinear form $(\cdot,\cdot)$, since it is simple and 
% $(V_L)_{0}/L(1)(V_L)_{1}$ is one-dimensional (\cite{Li}). We normalize this 
% form by setting $(\mathbf{1},\mathbf{1})=1$; then for any $u,v\in V_L$,
% \begin{equation*}
%  (u,v)=\mathrm{Res}_x\, x^{-1}(\mathbf{1}, 
% Y(e^{xL(1)}(-x^{-2})^{L(0)}u,x^{-1})v).
% \end{equation*}
% In particular, for $\alpha,\beta\in L$,
% \begin{align*}
%  (\iota(e_\alpha),\iota(e_\beta)) & =\mathrm{Res}_x \,
% x^{-1}(\mathbf{1},(-x^{-2})^{\langle\alpha,\alpha\rangle/2} 
% E^-(-\alpha,x^{-1})x^{\langle\alpha,\beta\rangle}\iota(e_\alpha e_\beta))\\
% & =(-1)^{\langle\alpha,\alpha\rangle/2} 
% \varepsilon(\alpha,\beta)\delta_{\alpha+\beta,0}(\mathbf{1},\iota(e_{
% \alpha+\beta}))=(-1)^{\langle\alpha,\alpha\rangle/2}\varepsilon(\alpha,
% \beta)\delta_{\alpha+\beta,0}.
% \end{align*}

\section{Integral intertwining operators among modules for lattice vertex 
operator algebras}\label{sec:lattintwops}

We continue to fix a nondegenerate even lattice $L$ in this section. We will 
classify the integral intertwining operators among $V_L$-modules, using the 
integral forms for $V_L$-modules that were introduced in \cite{M2}. (The 
integral form of $V_L$ itself was first introduced in \cite{B}.) We recall 
these 
integral forms by stating the following theorem from \cite{M2} (the algebra 
part 
was also proved in \cite{P} and \cite{DG}):
\begin{theo}\label{latttheo}
 The vertex subring $V_{L,\Z}$ of $V_L$ generated by the vectors 
$\iota(e_\alpha)$ for $\alpha\in L$ is an integral form of $V_L$.  Moreover, 
for 
any $\beta\in L^\circ$, the $V_{L,\Z}$-submodule $V_{\beta+L,\Z}$ of 
$V_{\beta+L}$ generated by $\iota(e_\beta)$ is an integral form of 
$V_{\beta+L}$.
\end{theo}
\begin{rema}\label{inteps}
 For this theorem to hold, we must assume that $\varepsilon$ satisfies 
$\varepsilon(\alpha,\beta)=\pm 1$ for any $\alpha,\beta\in L$. In fact the proof (although not the statement) of Lemma 4.1 in \cite{M2} shows that we can choose $\varepsilon$ to satisfy the stronger condition $\varepsilon(\alpha,\beta)=\pm 1$ whenever either $\alpha\in L$ or $\beta\in L$. 
\end{rema}
\begin{rema}\label{intformdes}
 The proof of Theorem \ref{latttheo} that is given in \cite{M2}, using the weaker condition in Remark \ref{inteps}, shows that for $\beta\in L^\circ$, $V_{\beta+L,\Z}$ is the $\Z$-span of coefficients of products of the form
 \begin{equation*}
  E^-(-\alpha_1,x_1)\cdots E^-(-\alpha_k,x_k)\iota(e_\alpha e_\beta)
 \end{equation*}
where $\alpha_1,\ldots,\alpha_k,\alpha\in L$. If we assume the stronger condition in Remark \ref{inteps}, then for any $\beta\in L^\circ$, $V_{\beta+L,\Z}$ is the $\Z$-span of coefficients of products of the form
 \begin{equation*}
  E^-(-\alpha_1,x_1)\cdots E^-(-\alpha_k,x_k)\iota(e_\gamma)
 \end{equation*}
where $\alpha_1,\ldots,\alpha_k\in L$ and $\gamma\in\beta+L$. Note that the proof of Theorem 7.1 which is given in \cite{M2} incorrectly states that the second description of $V_{\beta+L,\Z}$ follows from the weaker condition in Remark \ref{inteps}. In the rest of this section, we will generally use the first description of $V_{\beta+L,\Z}$ since some of the formulas we need are slightly simpler this way.
\end{rema}
\begin{rema}
 The integral form $V_{\beta+L,\Z}$ depends on the choice of $\varepsilon$, but once an $\varepsilon$ satisfying the stronger condition of Remark \ref{inteps} is fixed, the second description of $V_{\beta+L,\Z}$ in Remark \ref{intformdes} shows that it does not depend on the choice of generator. That is, the $V_{L,\Z}$-module in $V_{\beta+L}$ generated by $\iota(e_\beta)$ is the same as the $V_{L,\Z}$-module generated by $\iota(e_{\beta'})$ for any $\beta'\in\beta+L$.
\end{rema}

In order to obtain integral intertwining operators, we will need an explicit 
description, as in Remark \ref{intformdes}, of the graded $\mathbb{Z}$-dual of 
the integral form of a $V_L$-module. Recalling that the contragredient of 
$V_{\beta+L}$ is
$V_{-\beta+L}$, we shall use an invariant bilinear pairing as in 
(\ref{contform}) to identify $V_{-\beta+L,\mathbb{Z}}'$ as a sublattice of 
$V_{\beta+L,\mathbb{Z}}$. First we 
calculate (\ref{contform}) with $\mathcal{Y}=\mathcal{Y}_\beta$ and with the 
form $(\cdot,\cdot)_{V_L}$ on $V_L$ normalized so that 
$(\mathbf{1},\mathbf{1})_{V_L}=1$. Thus for $u\in V_{\beta+L}$ and $v\in 
V_{-\beta+L}$, we have
\begin{equation*}
 (u,v)=\mathrm{Res}_x (\mathbf{1},(\mathcal{Y}_\beta)^o_0(u,e^{\pi i} x) e^{x 
L(1)}v)_{V_L}.
\end{equation*}
In particular, (\ref{latintwop}) implies that for $\gamma\in\beta+L$, 
$\gamma'\in -\beta+L$,
\begin{align}\label{formform}
  (\iota & (e_\gamma),\iota(e_{\gamma'}))  \nonumber\\
& =\mathrm{Res}_x\, x^{-1} (\mathbf{1}, e^{-\pi i\langle\gamma,\gamma\rangle/2}  
x^{-\langle\gamma,\gamma\rangle} E^-(-\gamma,-x^{-1})e_\gamma (e^{\pi i} 
x)^{-\langle\gamma,\gamma'\rangle} e^{\pi i\langle\beta,\gamma'\rangle} 
c(\gamma',\beta)\iota(e_{\gamma'}))_{V_L}\nonumber\\
 & =\mathrm{Res}_x\, x^{-1} (\mathbf{1}, e^{\pi 
i(\langle\beta-\gamma,\gamma'\rangle-\langle\gamma,\gamma\rangle/2)} 
x^{-\langle\gamma,\gamma+\gamma'\rangle} E^-(-\gamma,-x^{-1}) 
c(\gamma',\beta)\varepsilon(\gamma,\gamma')\iota(e_{\gamma+\gamma'}))_{V_L}
\nonumber\\
 & =e^{\pi i\langle\gamma-2\beta,\gamma\rangle/2} 
c(\gamma,\beta)^{-1}\varepsilon(\gamma,\gamma)^{-1}\delta_{\gamma+\gamma',0}.
\end{align}
Using (\ref{formform}), we see that for any $\alpha\in L$,
\begin{align*}
 (\iota(e_\alpha e_\beta),\iota(e_{-\alpha} e_{-\beta})) & 
=\varepsilon(\alpha,\beta)\varepsilon(-\alpha,-\beta)(\iota(e_{\alpha+\beta}),
\iota(e_{-\alpha-\beta}))\\
 & =\varepsilon(\alpha,\beta)^2 e^{\pi 
i\langle\alpha-\beta,\alpha+\beta\rangle/2} 
c(\alpha+\beta,\beta)^{-1}\varepsilon(\alpha+\beta,\alpha+\beta)^{-1}\nonumber\\
 & =e^{\pi i\langle\alpha,\alpha\rangle/2}\varepsilon(\alpha,\alpha)^{-1} 
e^{-\pi i\langle\beta,\beta\rangle/2} 
c(\beta,\beta)^{-1}\varepsilon(\beta,\beta)^{-1},
\end{align*}
since $c(\alpha,\beta)=\varepsilon(\alpha,\beta)\varepsilon(\beta,\alpha)^{-1}$ 
from (\ref{commcocycle}).

If we now renormalize the invariant pairing $(\cdot,\cdot)$ between 
$V_{\beta+L}$ and $V_{-\beta+L}$ by setting
\begin{equation*}
 (\cdot,\cdot)_{\mathrm{new}}=e^{\pi i\langle\beta,\beta\rangle/2} 
c(\beta,\beta)\varepsilon(\beta,\beta) (\cdot,\cdot)_{\mathrm{old}},
\end{equation*}
we see that now
\begin{equation*}
 (\iota(e_\alpha e_\beta),\iota(e_{\alpha'} 
e_{-\beta}))=(-1)^{\langle\alpha,\alpha\rangle/2}\varepsilon(\alpha,\alpha)^{-1}
\delta_{\alpha+\alpha',0}=0\;\mathrm{or}\;\pm 1
\end{equation*}
for any $\alpha,\alpha'\in L$, using Remark \ref{inteps}. We can use this new 
invariant bilinear pairing 
between $V_{\beta+L}$ and $V_{-\beta+L}$, which is the form in (\ref{contform}) 
with
\begin{equation*}
 \mathcal{Y}=e^{\pi 
i\langle\beta,\beta\rangle/2}c(\beta,\beta)\varepsilon(\beta,\beta)\mathcal{Y}
_\beta\vert_{V_{\beta+L}\otimes V_{-\beta+L}},
\end{equation*}
to identify $V_{-\beta+L,\mathbb{Z}}'$ as a sublattice of 
$V_{\beta+L,\mathbb{Z}}$, for any $\beta\in L^\circ$:
\begin{propo}\label{lattcontmod}
 For any $\beta\in L^\circ$, the integral form $V_{-\beta+L,\mathbb{Z}}'$ is 
the 
integral form of $V_{\beta+L}$ integrally spanned by coefficients of products 
of 
the form
 \begin{equation}\label{dualspan}
  E^-(-\beta_1,x_1)\cdots E^-(-\beta_k,x_k)\iota(e_\alpha e_\beta)
 \end{equation}
where $\beta_i\in L^\circ$ and $\alpha\in L$.
\end{propo}
\begin{proof}
 Let $\widetilde{V}_{\beta+L,\Z}$ be the integral form of $V_{\beta+L}$ 
integrally spanned by coefficients of products as in \eqref{dualspan}. We first 
show that $\widetilde{V}_{\beta+L,\Z}\subseteq V_{-\beta+L,\Z}'$. For $\beta\in 
L^\circ$ and $n\in\Z$, it is easy to see from \eqref{oppopdef} that the adjoint 
of the operator $\beta(n)$ is $-\beta(-n)$. Thus the adjoint of
 \begin{equation*}
  E^-(-\beta,x)=\mathrm{exp}\left(\sum_{n<0}\dfrac{-\beta(n)}{n} x^{-n}\right)
 \end{equation*}
is
 \begin{equation*}
  \mathrm{exp}\left(\sum_{n<0}\dfrac{\beta(-n)}{n} 
x^{-n}\right)=\mathrm{exp}\left(\sum_{n>0}\dfrac{-\beta(n)}{n} 
x^n\right)=E^+(-\beta, x^{-1}).
 \end{equation*}
From Proposition 4.3.1 in \cite{FLM}, we have
\begin{equation*}
 E^+(\beta,x_1) E^-(\gamma,x_2) = E^-(\gamma,x_2) 
E^+(\beta,x_1)\left(1-\dfrac{x_2}{x_1}\right)^{\langle\beta,\gamma\rangle}
\end{equation*}
for any $\beta,\gamma\in L^\circ$. Thus for any $\beta_1,\ldots,\beta_k\in 
L^\circ$ and $\alpha_1,\ldots,\alpha_{k'},\alpha,\alpha'\in L$,
\begin{align}\label{formcalc}
 & \left( E^-(-\beta_1, x_1)\cdots E^-(-\beta_k, x_k)\iota(e_\alpha e_\beta), 
E^-(-\alpha_1, y_1)\cdots E^-(-\alpha_{k'}, y_{k'})\iota(e_{\alpha'} 
e_{-\beta})\right)\nonumber\\
 &\;\;\;\;\;\;\;\;  =\left(\iota(e_\alpha e_\beta), 
E^+(-\beta_1,x_1^{-1})\cdots E^+(-\beta_k,x_k^{-1}) E^-(-\alpha_1,y_1)\cdots 
E^-(-\alpha_{k'}, y_{k'}),\iota(e_{\alpha'} e_{-\beta})\right)\nonumber\\
 &\;\;\;\;\;\;\;\; =\left(\iota(e_\alpha e_\beta),E^-(-\alpha_1,y_1)\cdots 
E^-(-\alpha_{k'},y_{k'})\iota(e_{\alpha'} 
e_{-\beta})\right)\prod_{i=1}^k\prod_{j=1}^{k'} (1-x_i 
y_j)^{\langle\beta_i,\alpha_j\rangle}\nonumber\\
 &\;\;\;\;\;\;\;\; 
=(-1)^{\langle\alpha,\alpha\rangle/2}\varepsilon(\alpha,\alpha)^{-1}\delta_{
a+\alpha',0}\prod_{i=1}^k\prod_{j=1}^{k'} (1-x_i 
y_j)^{\langle\beta_i,\alpha_j\rangle}\in\Z[[\lbrace x_i\rbrace, \lbrace 
y_j\rbrace]]
\end{align}
since each $\langle\beta_i,\alpha_j\rangle\in\Z$. Since $V_{-\beta+L,\Z}$ is 
the integral span of coefficients of products of the form 
\begin{equation*}
  E^-(-\alpha_1, y_1)\cdots E^-(-\alpha_{k'}, y_{k'})\iota(e_{\alpha'} 
e_{-\beta})
\end{equation*}
for $\alpha_1,\ldots,\alpha_{k'},\alpha'\in L$ by Remark \ref{intformdes}, it 
follows that $\widetilde{V}_{\beta+L,\Z}\subseteq V_{-\beta+L,\Z}'$.

Now we must show the opposite inclusion
$V_{-\beta+L,\Z}'\subseteq\widetilde{V}_{\beta+L,\Z}$. For any $\beta\in L^\circ$ and 
$n>0$, let $s_{\beta,n}$ denote the coefficient of $x^n$ in $E^-(-\beta,x)$, 
and for any partition $\Lambda=(\lambda_1,\ldots,\lambda_k)$ where 
$\lambda_1\geq\ldots\geq\lambda_k>0$ we define
\begin{equation*}
 h_\Lambda(\beta)=s_{\beta,\lambda_1}\cdots s_{\beta,\lambda_k}.
\end{equation*}
If $\lbrace\alpha^{(1)},\ldots,\alpha^{(l)}\rbrace$ is a basis for $L$, then 
$V_{-\beta+L,\Z}$ has a basis consisting of the vectors
\begin{equation}\label{modbasis}
 h_{\Lambda_1}(\alpha^{(1)})\cdots h_{\Lambda_l}(\alpha^{(l)})\iota(e_{-\alpha} 
e_{-\beta})
\end{equation}
where $\alpha\in L$ and $\Lambda_1,\ldots,\Lambda_k$ run over all partitions 
(see \cite{M1} and \cite{DG}). 

Let $\lbrace\beta^{(1)},\ldots,\beta^{(l)}\rbrace$ be the basis of $L^\circ$ 
dual to $\lbrace\alpha^{(1)},\ldots,\alpha^{(l)}\rbrace$. Then in the basis of 
$V_{-\beta+L,\Z}'$ dual to \eqref{modbasis}, the basis vector corresponding to 
$h_\Lambda(\alpha^{(i)})\iota(e_{-\alpha} e_{-\beta})$ for any $i$, $\alpha\in 
L$, and partition $\Lambda$ is given by
\begin{equation*}
 (-1)^{\langle\alpha,\alpha\rangle/2}\varepsilon(\alpha,\alpha) 
m_\Lambda(\alpha^{(i)})\iota(e_{\alpha} e_\beta)
\end{equation*}
where $m_\Lambda(\alpha^{(i)})$ is a polynomial in the operators 
$\beta^{(i)}(-n)$, $n>0$. It follows that the basis vector dual to the vector 
\eqref{modbasis} is
\begin{equation*}
 (-1)^{\langle\alpha,\alpha\rangle/2}\varepsilon(\alpha,\alpha) 
m_{\Lambda_1}(\alpha^{(1)})\cdots m_{\Lambda_k}(\alpha^{(l)})\iota(e_\alpha 
e_\beta).
\end{equation*}
Thus it is enough to show that for any $i$ and partition $\Lambda$, 
$m_\Lambda(\alpha^{(i)})$ is an integer linear combination of operators 
$h_{\Lambda'}(-\beta^{(i)})$ for partitions $\Lambda'$.

We will need to use the well-known monomial and complete symmetric polynomials 
(see for example \cite{Mac} Chapter 1, Section 2). Fix a positive integer $n$. 
Then for any partition $\Lambda=(\lambda_1,\ldots,\lambda_k)$ where $k\leq n$, 
the monomial symmetric polynomial in $n$ variables $m_\Lambda(y_1,\ldots, y_n)$ 
is the sum of all distinct permutations of the monomial $y_1^{\lambda_1}\cdots 
y_k^{\lambda_k}$. The polynomials $m_\Lambda(y_1,\ldots,y_n)$ as $\Lambda$ runs 
over partitions with less than or equal to $n$ parts form a basis of the 
(graded) ring of $n$-variable symmetric polynomials with integer coefficients. 
On the other hand, for any $m>0$, the complete homogeneous symmetric polynomial 
$h_m(y_1,\ldots,y_n)$ is the sum of all distinct degree $m$ monomials in $n$ 
variables; for any partition $\Lambda=(\lambda_1,\ldots,\lambda_k)$, we define 
the polynomial
\begin{equation*}
 h_\Lambda(y_1,\ldots,y_n)=h_{\lambda_1}(y_1,\ldots,y_n)\cdots 
h_{\lambda_k}(y_1,\ldots,y_n).
\end{equation*}
The polynomials $h_\Lambda(y_1,\ldots, y_n)$ as $\Lambda$ runs over all 
partitions with parts less than or equal to $n$ forms another basis for the 
ring of $n$-variable symmetric polynomials. In particular, for any partition 
$\Lambda$ of $n$,
\begin{equation}\label{handm}
 m_\Lambda(y_1,\ldots,y_n)=\sum k_{\Lambda'} h_{\Lambda'}(y_1,\ldots,y_n)
\end{equation}
where $k_{\Lambda'}\in\Z$ and $\Lambda'$ runs over all partitions of $n$.

In light of \eqref{handm}, we will show that for any partition 
$\Lambda=(\lambda_1,\ldots,\lambda_k)$ of an integer $n$, any $i$, and 
$\alpha\in L$,
\begin{align}\label{mform}
 & \left(m_\Lambda(\alpha^{(i)})\iota(e_\alpha e_\beta), E^-(-\alpha^{(i)}, 
y_1)\cdots E^-(-\alpha^{(i)}, y_n)\iota(e_{-\alpha} 
e_{-\beta})\right)\nonumber\\ 
&\;\;\;\;\;\;\;\;\;\;\;\;\;\;\;\;=(-1)^{\langle\alpha,\alpha\rangle/2}
\varepsilon(\alpha,\alpha)^{-1} m_\Lambda(y_1,\ldots, y_n)
\end{align}
and
\begin{align}\label{hform}
& \left(h_\Lambda(-\beta^{(i)})\iota(e_\alpha e_\beta), E^-(-\alpha^{(i)}, 
y_1)\cdots E^-(-\alpha^{(i)}, y_n)\iota(e_{-\alpha} 
e_{-\beta})\right)\nonumber\\
&\;\;\;\;\;\;\;\;\;\;\;\;\;\;\;\;=(-1)^{\langle\alpha,\alpha\rangle/2}
\varepsilon(\alpha,\alpha)^{-1} h_\Lambda(y_1,\ldots, y_n).
\end{align}
Since $m_\Lambda(\alpha^{(i)})$ and $h_\Lambda(-\beta^{(i)})$ are operators of 
degree $n$ and are polynomials in the operators $\beta^{(i)}(-m)$ for $m>0$, 
$V_{-\beta+L,\Z}$ is spanned by coefficients of monomials in
\begin{equation}\label{alphaispan}
 E^-(-\alpha^{(i)}, y_1)\cdots E^-(-\alpha^{(i)}, y_n)\iota(e_{-\alpha} 
e_{-\beta})
\end{equation}
and by vectors $v$ such that
\begin{equation*}
 \left(m_\Lambda(\alpha^{(i)})\iota(e_\alpha 
e_\beta),v\right)=\left(h_\Lambda(-\beta^{(i)})\iota(e_\alpha e_\beta), 
v\right) =0.
\end{equation*}
Thus \eqref{handm}, \eqref{mform}, and \eqref{hform} imply that
\begin{equation*}
 m_\Lambda(\alpha^{(i)})\iota(e_\alpha e_\beta)=\sum k_{\Lambda'} 
h_{\Lambda'}(-\beta^{(i)})\iota(e_\alpha e_\beta),
\end{equation*}
and thus
\begin{equation*}
 m_{\Lambda}(\alpha^{(i)})=\sum k_{\Lambda'} h_{\Lambda'}(-\beta^{(i)}),
\end{equation*}
which will complete the proof.

To verify \eqref{mform}, we observe that by the definition of 
$m_\Lambda(\alpha^{(i)})$, a monomial on the left side of \eqref{mform} has 
non-zero coefficient (equal to 
$(-1)^{\langle\alpha,\alpha\rangle/2}\varepsilon(\alpha,\alpha)^{-1}$) if and 
only if the corresponding coefficient of \eqref{alphaispan} is 
$h_\Lambda(\alpha^{(i)})\iota(e_{-\alpha} e_{-\beta})$. Since the coefficient 
of a monomial in \eqref{alphaispan} equals 
$h_\Lambda(\alpha^{(i)})\iota(e_{-\alpha} e_{-\beta})$ if and only if the 
monomial is a permutation of $y_1^{\lambda_1}\cdots y_k^{\lambda_k}$, 
\eqref{mform} follows from the definition of $m_\Lambda(y_1,\ldots,y_n)$.

To verify \eqref{hform}, we see from \eqref{formcalc} that the left side of 
\eqref{hform} is the coefficient of $x_1^{\lambda_1}\cdots x_k^{\lambda_k}$ in
\begin{align*} 
&(-1)^{\langle\alpha,\alpha\rangle/2}\varepsilon(\alpha,\alpha)^{-1}\prod_{i=1}
^k\prod_{j=1}^n (1-x_i y_j)^{-1} 
=(-1)^{\langle\alpha,\alpha\rangle/2}\varepsilon(\alpha,\alpha)^{-1}\prod_{i=1}
^k\prod_{j=1}^n\sum_{m\geq 0} x_i^m y_j^m\nonumber\\
 &\;\;\;\;\;\;\;\; 
=(-1)^{\langle\alpha,\alpha\rangle/2}\varepsilon(\alpha,\alpha)^{-1}\prod_{i=1}
^k\sum_{m\geq 0}\left(\sum_{\substack{m_1,\ldots, m_n\geq 0\\ 
m_1+\cdots+m_n=m}} y_1^{m_1}\cdots y_n^{m_n}\right) x_i^m\nonumber\\
 &\;\;\;\;\;\;\;\; 
=(-1)^{\langle\alpha,\alpha\rangle/2}\varepsilon(\alpha,\alpha)^{-1}\prod_{i=1}
^k\sum_{m\geq 0} h_m(y_1,\ldots,y_n)x_i^m.
\end{align*}
Thus the coefficient of $x_1^{\lambda_1}\cdots x_k^{\lambda_k}$ is
\begin{equation*}
 (-1)^{\langle\alpha,\alpha\rangle/2}\varepsilon(\alpha,\alpha)^{-1} 
h_\Lambda(y_1,\ldots, y_n)
\end{equation*}
as desired.
\end{proof}

Note that since $L\subseteq L^\circ$, Proposition \ref{lattcontmod} shows that 
$V_{-\beta+L,\mathbb{Z}}'$ is generally larger than $V_{\beta+L,\mathbb{Z}}$. 
We record the case $\beta =0$ of Proposition \ref{lattcontmod}:
\begin{corol}\label{lattcontalg}
 The integral form $V_{L,\mathbb{Z}}'$ is the integral span of coefficients of 
products of the form
 \begin{equation*}
   E^-(-\beta_1,x_1)\cdots E^-(-\beta_k,x_k)\iota(e_{\alpha})
 \end{equation*}
where $\beta_i\in L^\circ$ and $\alpha\in L$.
\end{corol}
\begin{rema}
 Corollary \ref{lattcontalg} provides an alternative to the description of $V_{L,\Z}'$ given in Proposition 3.6 of \cite{DG}.
\end{rema}

\begin{rema}
 Note that the precise identity of $V_{-\beta+L,\mathbb{Z}}'$ depends on the 
choice of normalization of the invariant bilinear pairing between $V_{\beta+L}$ 
and $V_{-\beta+L}$. We shall take $V_{-\beta+L,\mathbb{Z}}'$ as described in 
Proposition \ref{lattcontmod} as the official $V_{-\beta+L,\mathbb{Z}}'$, since 
we have $V_{\beta+L,\mathbb{Z}}\subseteq V_{-\beta+L,\mathbb{Z}}'$ this way.
\end{rema}

Finally we can classify some integral intertwining operators:
\begin{theo}\label{lattintwopth}
 For any $\beta,\gamma\in L^\circ$, there is a rank one lattice of intertwining 
operators within $V^{\beta+\gamma}_{\beta\,\gamma}$ integral with respect to 
$V_{\beta+L,\mathbb{Z}}$, $V_{\gamma+L,\mathbb{Z}}$, and 
$V_{-\beta-\gamma+L,\mathbb{Z}}'$. Moreover, this lattice is spanned by
 \begin{equation*}
  \mathcal{Y}_{\beta,\gamma,\mathbb{Z}}=e^{-\pi 
i\langle\beta,\gamma\rangle}\varepsilon(\gamma,\beta)^{-1}\mathcal{Y}
_\beta\vert_{V_{\beta+L}\otimes V_{\gamma+L}}.
 \end{equation*}
\end{theo}
\begin{proof}
From the definition (\ref{latintwop}) of $\mathcal{Y}_\beta$, we have
\begin{align*}
\mathcal{Y}_\beta(\iota(e_\beta),x)\iota(e_\gamma) & 
=E^-(-\beta,x)E^+(-\beta,x) 
x^{\langle\beta,\gamma\rangle} e^{\pi 
i\langle\beta,\gamma\rangle}c(\gamma,\beta)\iota(e_\beta e_\gamma)\nonumber\\
& =x^{\langle\beta,\gamma\rangle} E^-(-\beta,x) e^{\pi 
i\langle\beta,\gamma\rangle}\varepsilon(\gamma,\beta)\varepsilon(\beta,\gamma)^{
-1}\varepsilon(\beta,\gamma)\iota(e_{\beta+\gamma})\nonumber\\
& = x^{\langle\beta,\gamma\rangle} E^-(-\beta,x) e^{\pi 
i\langle\beta,\gamma\rangle} \varepsilon(\gamma,\beta)\iota(e_{\beta+\gamma}).
\end{align*}
Then
\begin{equation}\label{betagammaop}
 \mathcal{Y}_{\beta,\gamma,\mathbb{Z}}(\iota(e_\beta),x)\iota(e_\gamma)= 
x^{\langle\beta,\gamma\rangle} E^-(-\beta,x)\iota(e_{\beta+\gamma})\in 
V_{-\beta-\gamma+L,\mathbb{Z}}'\lbrace x\rbrace
\end{equation}
by Proposition \ref{lattcontmod}.

Since $\iota(e_\beta)$ and $\iota(e_\gamma)$ generate $V_{\beta+L,\mathbb{Z}}$ 
and $V_{\gamma+L,\mathbb{Z}}$, respectively, as $V_{L,\mathbb{Z}}$-modules, 
Theorem \ref{intwopgen} implies that $\mathcal{Y}_{\beta,\gamma,\mathbb{Z}}$ is 
integral with respect to $V_{\beta+L,\mathbb{Z}}$, $V_{\gamma+L,\mathbb{Z}}$, 
and $V_{-\beta-\gamma+L,\mathbb{Z}}'$. Moreover, we see from Proposition 
\ref{lattcontmod} and (\ref{betagammaop}) that for $c\in\mathbb{C}$,
\begin{equation*}
 c\mathcal{Y}_{\beta,\gamma,\mathbb{Z}}(\iota(e_\beta),x)\iota(e_\gamma)\in 
V_{-\beta-\gamma+L,\mathbb{Z}}'\lbrace x\rbrace
\end{equation*}
if and only if $c\in\mathbb{Z}$. Thus $\mathcal{Y}_{\beta,\gamma,\mathbb{Z}}$ 
spans the lattice of intertwining operators in 
$V^{\beta+\gamma}_{\beta\,\gamma}$ which are integral with respect to 
$V_{\beta+L,\mathbb{Z}}$, $V_{\gamma+L,\mathbb{Z}}$, and 
$V_{-\beta-\gamma+L,\mathbb{Z}}'$.
\end{proof}

\begin{rema}
 Note that to prove Theorem \ref{lattintwopth} we only need the ``easy'' half of Proposition \ref{lattcontmod} which states that for any $\beta\in L^\circ$, the coefficients of products as in \eqref{dualspan} are contained in $V_{-\beta+L,\Z}'$. But the more difficult opposite inclusion gives the additional information that for $\beta,\gamma\in L^\circ$, the coefficients of $\mathcal{Y}_{\beta,\gamma,\Z}(w_\beta,x)w_\gamma$ for any $w_\beta\in V_{\beta+L,\Z}$, $w_\gamma\in W_{\gamma+L,\Z}$ are integral linear combinations of the coefficients of products as in \eqref{dualspan} (with $\beta$ replaced by $\beta+\gamma$).
\end{rema}

\begin{rema}
 Note that (\ref{betagammaop}) shows that for $\beta,\gamma\in L^\circ$, 
$\mathcal{Y}_{\beta,\gamma,\mathbb{Z}}(\iota(e_\beta),x)\iota(e_\gamma)\notin 
V_{\beta+\gamma+L,\mathbb{Z}}\lbrace x\rbrace$ since in general $\beta\notin 
L$. 
It is not clear that any non-zero integer multiple of 
$\mathcal{Y}_{\beta,\gamma,\mathbb{Z}}$ is in general integral with respect to 
$V_{\beta+L,\mathbb{Z}}$, $V_{\gamma+L,\mathbb{Z}}$, and 
$V_{\beta+\gamma+L,\mathbb{Z}}$.
\end{rema}

\noindent{\small \sc 
Beijing International Center for Mathematical Research, Peking University, 
Beijing, China 100084}\\
{\em E-mail address}:
\texttt{robertmacrae@math.pku.edu.cn} \\

\end{document}